\documentclass[reqno]{amsart}

\usepackage{amsfonts,amssymb,latexsym,amsmath,amsxtra,tikz}
\usepackage{verbatim}

\newcommand{\field}[1]{\mathbb{#1}}

\newcommand{\N}{\field{N}}

\newcommand{\ov}{\overline}

\numberwithin{equation}{section}
\newtheorem{theorem}{Theorem}[section]
\newtheorem{lemma}[theorem]{Lemma}
\newtheorem{corollary}[theorem]{Corollary}

\newtheorem{proposition}[theorem]{Proposition}
\theoremstyle{remark}

\renewenvironment{proof}[1][Proof]{\begin{trivlist}
\item[\hskip \labelsep {\bfseries #1:}]}{\qed\end{trivlist}}

\title[Generalisations of Schur's theorem]{Unification, refinements and companions of generalisations of Schur's theorem}
\author{Jehanne Dousse}
\address{Institut f\"ur Mathematik, Universit\"at Z\"urich\\ Winterthurerstrasse 190, 8057 Z\"urich, Switzerland}
\email{jehanne.dousse@math.uzh.ch}
\allowdisplaybreaks

\dedicatory{To Krishna Alladi in honour of his 60th birthday}

\begin{document}


\begin{abstract}
We prove a general theorem on overpartitions with difference conditions that unifies generalisations of Schur's theorem due to Alladi-Gordon, Andrews, Corteel-Lovejoy and the author. This theorem also allows one to give companions and refinements of the generalisations of Andrews' theorems to overpartitions. The proof relies on the method of weighted words of Alladi and Gordon and $q$-difference equation techniques introduced recently by the author.
\end{abstract}

\maketitle

%
%

\section{Introduction}
A partition of $n$ is a non-increasing sequence of natural numbers whose sum is $n$.
In $1926$, Schur \cite{Schur} proved the following partition identity.
\begin{theorem}[Schur]
\label{schur}
For any integer $n$, let $A(n)$ denote the number of partitions of $n$ into distinct parts congruent to $1$ or $2$ modulo $3$, and $B(n)$ the number of partitions of $n$ such that parts differ by at least $3$ and no two consecutive multiples of $3$ appear. Then for all $n$,
$$A(n)=B(n).$$
\end{theorem}

Schur's theorem became very influential and several proofs have been given using a variety of different techniques~\cite{Alladi,Andrews2,Andrews1,Andrews3,Bessenrodt,Bressoud}. For our purposes in this article, the most significant proofs are a proof of Alladi and Gordon~\cite{Alladi} using the method of weighted words and two proofs of Andrews~\cite{Andrews2,Andrews1} using recurrences and $q$-difference equations.

The idea of the method of weighted words of Alladi and Gordon is to give a combinatorial interpretation of the infinite product
$$ \prod_{n \geq 1} (1+aq^n)(1+bq^n)$$
as the generating function for partitions whose parts appear in three colours $a,b,ab.$

More precisely, they consider the following ordering of colours
\begin{equation} \label{colororder}
ab < a < b,
\end{equation}
giving the following ordering on coloured positive integers
$$1_{ab} < 1_a < 1_b < 2_{ab} < 2_a < 2_b < \cdots.$$
Denoting by $c(\lambda)$ the colour of $\lambda$, their refinement of Schur's theorem can be stated as follows.
\begin{theorem}[Alladi-Gordon]
\label{schurnondil}
Let $A(u,v,n)$ be the number of partitions of $n$ into $u$ distinct parts coloured $a$ and $v$ distinct parts coloured $b$.

Let $B(u,v,n)$ be the number of partitions $\lambda_1 + \cdots + \lambda_s$ of $n$ into distinct parts with no part $1_{ab}$, such that the difference $\lambda_i - \lambda_{i+1} \geq 2$ if $c(\lambda_{i}) = ab$ or $c(\lambda_i) < c(\lambda_{i+1})$ in \eqref{colororder}, having $u$ parts $a$ or $ab$ and $v$ parts $b$ or $ab$.

Then 
$$\sum_{u,v,n \geq 0} A(u,v,n) a^u b^vq^n = \sum_{u,v,n \geq 0} B(u,v,n) a^u b^vq^n = \prod_{n \geq 1} (1+aq^n)(1+bq^n).$$
\end{theorem}

Doing the transformations
$$q \rightarrow q^3, a \rightarrow a q^{-2}, b \rightarrow bq^{-1},$$
one obtains a refinement of Schur's theorem. For details, see~\cite{Alladi}.

On the other hand, using the ideas of his proofs with $q$-difference equations~\cite{Andrews2,Andrews1}, Andrews was able to generalise Schur's theorem in two different ways \cite{Generalisation2,Generalisation1}. Let us now recall some notation due to Andrews in order to state his generalisations.

Let $A=\lbrace a(1),\dots, a(r) \rbrace$ be a set of $r$ distinct positive integers such that $\sum_{i=1}^{k-1} a(i) < a(k)$ for all $1 \leq k \leq r$. Note that the $2^r -1$ possible sums of distinct elements of $A$ are all distinct. We denote this set of sums by $A'=\lbrace \alpha(1),\dots, \alpha(2^r -1) \rbrace$, where $\alpha(1) < \cdots < \alpha(2^r-1)$.
Let $N$ be a positive integer with $N \geq \alpha(2^r-1) = a(1) +\cdots+a(r).$ We further define $\alpha(2^r)=a(r+1)=N+a(1).$ Let $A_N$ (resp. $-A_N$) denote the set of positive integers congruent to some $a(i) \mod N$ (resp. $-a(i) \mod N$), $A'_N$ (resp. $-A'_N$) the set of positive integers congruent to some $\alpha(i) \mod N$ (resp. $-\alpha(i) \mod N$). Let $\beta_N(m)$ be the least positive residue of $m \mod N$. If $\alpha \in A'$, let $w_A(\alpha)$ be the number of terms appearing in the defining sum of $\alpha$ and $v_A(\alpha)$ (resp. $z_A(\alpha)$) the smallest (resp. the largest) $a(i)$ appearing in this sum.

The simplest example is the one where $a(k)=2^{k-1}$ for $1 \leq k \leq r$ and $\alpha(k)=k$ for $1 \leq k \leq 2^r-1$.

\begin{theorem}[Andrews]
\label{andrews}
Let $D(A_N;n)$ denote the number of partitions of $n$ into distinct parts taken from $A_N$. Let $E(A'_N;n)$ denote the number of partitions of $n$ into parts taken from $A'_N$ of the form $n=\lambda_1+\cdots+ \lambda_s$, such that
\[\lambda_i - \lambda_{i+1} \geq N w_A(\beta_N(\lambda_{i+1}))+v_A(\beta_N(\lambda_{i+1}))-\beta_N(\lambda_{i+1}).\]
Then for all $n \geq 0$, $$D(A_N;n)= E(A'_N;n).$$
\end{theorem}

\begin{theorem}[Andrews]
\label{andrews2}
Let $F(-A_N;n)$ denote the number of partitions of $n$ into distinct parts taken from $-A_N$. Let $G(-A'_N;n)$ denote the number of partitions of $n$ into parts taken from $-A'_N$ of the form $n=\lambda_1+\cdots+ \lambda_s$, such that
$$\lambda_i - \lambda_{i+1} \geq N w_A(\beta_N(-\lambda_{i}))+v_A(\beta_N(-\lambda_{i}))-\beta_N(-\lambda_{i}),$$
and
$$\lambda_s \geq N (w_A(\beta_N(-\lambda_s)-1).$$
Then for all $n \geq 0$, $$F(-A_N;n)= G(-A'_N;n).$$
\end{theorem}

Schur's theorem corresponds to the case $N=3$, $r=2$, $a(1)=1, a(2)=2$.

Andrews' identities led to a number of important developments in combinatorics~\cite{Alladi1,Corteel,Yee}, group representation theory~\cite{AndrewsOlsson} and quantum algebra~\cite{Oh}.

In 2006, Corteel and Lovejoy~\cite{Corteel} combined the ideas of Alladi-Gordon and Andrews to prove a general theorem on coloured partitions which unifies and refines Andrews' two hierarchies of partition identities. To state their refinement (slightly reformulated to fit our purposes), we need to introduce some more notation.

Let $r$ be a positive integer. We define $r$ primary colours $u_1, \dots, u_r$ and use them to define $2^r-1$ colours $\tilde{u}_1, \dots , \tilde{u}_{2^r-1}$ as follows:
$$ \tilde{u}_i := u_1^{\epsilon_1(i)} \cdots u_r^{\epsilon_r(i)},$$
where
$$\epsilon_k(i) := \begin{cases}
1 \text{ if } 2^{k-1} \text{ appears in the binary expansion of } i \\
0 \text{ otherwise.}
\end{cases}$$
They are ordered in the natural ordering, namely
$$\tilde{u}_1 < \cdots < \tilde{u}_{2^r-1}.$$
Now for all $i \in \{1, \dots, 2^r-1 \},$ let $v(\tilde{u}_i)$ (resp. $z(\tilde{u}_i)$) be the smallest (resp. largest) primary colour appearing in the colour $\tilde{u}_i$ and $w(\tilde{u}_i)$ be the number of primary colours appearing in $\tilde{u}_i$. Finally, for $i,j \in \{1, \dots, 2^r-1 \},$ let 
$$\delta(\tilde{u}_i, \tilde{u}_j) :=  \begin{cases}
1 \text{ if } z(\tilde{u}_i) < v(\tilde{u}_j)\\
0 \text{ otherwise.}
\end{cases}$$

In a slightly modified version, Corteel and Lovejoy's theorem may be stated as follows.

\begin{theorem}[Corteel-Lovejoy]
\label{th:corlov}
Let $D(\ell_1, \dots, \ell_r;n)$ denote the number of partitions of $n$ into distinct non-negative parts, each part being coloured in one of the primary colours $u_1, \dots, u_r$, having $\ell_i$ parts coloured $u_i$ for all $i \in \{1, \dots,r\}$. Let $E(\ell_1, \dots, \ell_r;n)$ denote the number of partitions $\lambda_1+\cdots+ \lambda_s$ of $n$ into distinct non-negative parts, each part being coloured in one of the colours $\tilde{u}_1, \dots, \tilde{u}_{2^r-1}$, such that for all  $i \in \{1, \dots,r\}$, $\ell_i$ parts have $u_i$ as one of their primary colours, satisfying the difference conditions
$$\lambda_i - \lambda_{i+1} \geq w(c(\lambda_{i+1}))+\delta(c(\lambda_i),c(\lambda_{i+1})).$$
Then for all $\ell_1, \dots, \ell_r,n \geq 0$,
$$D(\ell_1, \dots, \ell_r;n)= E(\ell_1, \dots, \ell_r;n).$$
\end{theorem}
The proof of Theorem~\ref{th:corlov} relies on the iteration of a bijection originally discovered by Bressoud~\cite{Bressoud2} and adapted by Alladi and Gordon to the context of weighted words~\cite{Alladi2}.

Corteel and Lovejoy then noticed that the partitions counted by $D(\ell_1, \dots, \ell_r;n)$ and $E(\ell_1, \dots, \ell_r;n)$ have some symmetry properties and took advantage of them to prove an even more general theorem.

Let $\sigma \in S_r$ be a permutation. For every colour $\tilde{u}_i = u_1^{\epsilon_1(i)} \cdots u_r^{\epsilon_r(i)}$, we define the colour
$$\sigma(\tilde{u}_i):= u_{\sigma(1)}^{\epsilon_1(i)} \cdots u_{\sigma(r)}^{\epsilon_r(i)}.$$

Now for every partition $\lambda$ counted by $E(\ell_1, \dots, \ell_r;n)$, we define a new partition $\lambda^{\sigma}$ obtained by setting $\lambda_i^{\sigma} = \lambda_i$ and $c(\lambda_i^{\sigma})=\sigma(c(\lambda_i))$. This mapping is easily reversible by using the inverse permutation $\sigma^{-1}$ on $\lambda^{\sigma}$.
This transformation doesn't change $w(c(\lambda_{i+1}))$, so the difference condition we obtain on $\lambda^{\sigma}$ is
\begin{equation}
\label{cond_diff_sigma}
\lambda_i^{\sigma} - \lambda_{i+1}^{\sigma} \geq w(c(\lambda_{i+1}^{\sigma}))+\delta(\sigma^{-1}(c(\lambda_i^{\sigma})),\sigma^{-1}(c(\lambda_{i+1}^{\sigma}))).
\end{equation}
Thus $E(\ell_1, \dots, \ell_r;n) = E^{\sigma}(\ell_{\sigma^{-1}(1)}, \dots, \ell_{\sigma^{-1}(r)};n)$, where $E^{\sigma}(\ell_1, \dots, \ell_r;n)$ denotes the number of partitions of $n$ into distinct non-negative parts, each part being coloured in one of the colours $\tilde{u}_1, \dots, \tilde{u}_{2^r-1}$, such that for all  $i \in \{1, \dots,r\}$, $\ell_i$ parts have $u_i$ as one of their primary colours, satisfying the difference condition~\eqref{cond_diff_sigma}.

Moreover, by doing the same transformation on the partitions counted by $D(\ell_1, \dots, \ell_r;n),$ one can see that
$$D(\ell_1, \dots, \ell_r;n)=D(\ell_{\sigma^{-1}(1)}, \dots, \ell_{\sigma^{-1}(r)};n).$$

Thus one has
\begin{corollary}[Corteel-Lovejoy]
For every permutation $\sigma \in S_r$,
$$D(\ell_{\sigma^{-1}(1)}, \dots, \ell_{\sigma^{-1}(r)};n)=E^{\sigma}(\ell_{\sigma^{-1}(1)}, \dots, \ell_{\sigma^{-1}(r)};n).$$
\end{corollary}

One obtains a refinement of Theorem~\ref{andrews} by using the permutation $\sigma=Id$ and doing the transformations
$$ q \rightarrow q^N, u_1 \rightarrow u_1q^{a(1)}, \dots u_r \rightarrow u_r q^{a(r)},$$
and a refinement of Theorem~\ref{andrews2} by using the permutation $\sigma=(n,n-1, \dots,1)$ and doing the transformations
$$ q \rightarrow q^N, u_1 \rightarrow u_1q^{N-a(1)}, \dots u_r \rightarrow u_r q^{N-a(r)}.$$
More detail on how to recover Andrews' theorems is given in Section 2 in the case of overpartitions, which generalises the case of partitions.

\ \\

Let us now mention the extensions of Schur's theorem and its generalisations to overpartitions.
An overpartition of $n$ is a partition of $n$ in which the first occurrence of a number may be overlined.
For example, there are $14$ overpartitions of $4$:
$4$, $\overline{4}$, $3+1$, $\overline{3}+1$, $3+\overline{1}$, $\overline{3}+\overline{1}$, $2+2$, $\overline{2}+2$, $2+1+1$, $\overline{2}+1+1$, $2+\overline{1}+1$, $\overline{2}+\overline{1}+1$, $1+1+1+1$ and $\overline{1}+1+1+1$.
Though they were not called overpartitions at the time, they were already used in 1967 by Andrews~\cite{Andrews67} to give combinatorial interpretations of the $q$-binomial theorem, Heine's transformation and Lebesgue's identity. Then they were used in 1987 by Joichi and Stanton~\cite{JoichiStanton} in an algorithmic theory of bijective proofs of $q$-series identities. They also appear in bijective proofs of Ramanujan's $_1\psi_1$ summation and the $q$-Gauss summation~\cite{Corteel03,CorLov02}. It was Corteel~\cite{Corteel03} who gave them their name in 2003, just before Corteel and Lovejoy~\cite{CorLov} revealed their generality by giving combinatorial interpretations for several $q$-series identities.
They went on to become a very interesting generalisation of partitions, and several partition identities have overpartition analogues or generalisations. For example, Lovejoy proved overpartition analogues of identities of Gordon~\cite{Lovejoy03}, and Andrews-Santos and Gordon-G\"ollnitz~\cite{Lovejoy04}. Overpartitions also have interesting arithmetic properties~\cite{acko, BL2, Ma, Tre} and are related to the fields of Lie algebras~\cite{Kang}, mathematical physics~\cite{Desrosiers,Fortin1,Fortin2} and supersymmetric functions \cite{Desrosiers}.

In 2005, Lovejoy~\cite{Lovejoy} generalised Schur's theorem (in the weighted words version) to overpartitions by proving the following.

\begin{theorem}[Lovejoy]
\label{schur_over}
Let $\overline{A}(x_1,x_2;k,n)$ denote the number of overpartitions of $n$ into $x_1$ parts congruent to $1$ and $x_2$ parts congruent to $2$ modulo $3$, having $k$ non-overlined parts.
Let $\overline{B}(x_1,x_2;k,n)$ denote the number of overpartitions $\lambda_1 + \cdots + \lambda_s$ of $n$, with $x_1$ parts congruent to $0$ or $1$ modulo $3$ and $x_2$ parts congruent to $0$ or $2$ modulo $3$, having $k$ non-overlined parts and satisfying the difference conditions
\begin{equation*}
\lambda_i - \lambda_{i+1} \geq 
\begin{cases}
0 + 3 \chi(\overline{\lambda_{i+1}})\ \text{if}\ \lambda_{i+1} \equiv 1,2 \mod 3,\\
1 + 3 \chi(\overline{\lambda_{i+1}})\ \text{if}\ \lambda_{i+1} \equiv 0 \mod 3,
\end{cases}
\end{equation*}
where $\chi(\overline{\lambda_{i+1}}) =1$ if $\lambda_{i+1}$ is overlined and $0$ otherwise.
Then for all $x_1,x_2,k,n \geq 0$, $\overline{A}(x_1,x_2;k,n)=\overline{B}(x_1,x_2;k,n)$.
\end{theorem}

Schur's theorem (in the refined version of Alladi and Gordon) corresponds to the case $k=0$ in Lovejoy's theorem.

Recently, the author generalised both of Andrews' theorems (Theorems~\ref{andrews} and~\ref{andrews2}) to overpartitions~\cite{Doussegene,Doussegene2} by proving the following (reusing the notation of Andrews' theorems).

\begin{theorem}[Dousse]
\label{dousse}
Let $\overline{D}(A_N;k,n)$ denote the number of overpartitions of $n$ into parts taken from $A_N$, having $k$ non-overlined parts. Let $\overline{E}(A'_N;k,n)$ denote the number of overpartitions of $n$ into parts taken from $A'_N$ of the form $n=\lambda_1+\cdots+ \lambda_s$, having $k$ non-overlined parts, such that
\[\lambda_i - \lambda_{i+1} \geq N \left(w_A\left(\beta_N(\lambda_{i+1})\right) -1 +\chi(\ov{\lambda_{i+1}}) \right)+v_A(\beta_N(\lambda_{i+1}))-\beta_N(\lambda_{i+1}),\]
where $\chi(\ov{\lambda_{i+1}})=1$ if $\lambda_{i+1}$ is overlined and $0$ otherwise.
Then for all $k,n \geq 0$, $\overline{D}(A_N;k,n)= \overline{E}(A'_N;k,n)$.
\end{theorem}

\begin{theorem}[Dousse]
\label{dousse2}
Let $\overline{F}(-A_N;k,n)$ denote the number of overpartitions of $n$ into parts taken from $-A_N$, having $k$ non-overlined parts. Let $\overline{G}(-A'_N;k,n)$ denote the number of overpartitions of $n$ into parts taken from $-A'_N$ of the form $n=\lambda_1+\cdots+ \lambda_s$, having $k$ non-overlined parts, such that
$$
\lambda_i - \lambda_{i+1} \geq N \left(w_A\left(\beta_N(-\lambda_{i})\right) -1 +\chi(\ov{\lambda_{i+1}}) \right)+v_A(\beta_N(-\lambda_{i}))-\beta_N(-\lambda_{i}),
$$
and
$$
\lambda_s \geq N (w_A(\beta_N(-\lambda_s))-1).
$$
Then for all $k,n \geq 0$, $\overline{F}(-A_N;k,n)=\overline{G}(-A'_N;k,n)$.
\end{theorem}
Lovejoy's theorem corresponds to $N=3$, $r=2$, $a(1)= 1$, $a(2)=2$ in Theorems~\ref{dousse} and~\ref{dousse2}, while the case $k=0$ of Theorem~\ref{dousse} (resp. Theorem~\ref{dousse2}) gives Andrews' Theorem~\ref{andrews} (resp. Theorem~\ref{andrews2}).

While the statements of Theorems~\ref{dousse} and~\ref{dousse2} resemble those of Andrews' theorems (Theorem~\ref{andrews} and~\ref{andrews2}), the proofs are more intricate. We used $q$-difference equations and recurrences as well, but in our case we had equations of order $r$ while those of Andrews' proofs were easily reducible to equations of order $1$. Thus we needed to prove the result by induction on $r$ by going back and forth from $q$-difference equations on generating functions to recurrence equations on their coefficients.

The purpose of this paper is to generalise and refine Theorems~\ref{dousse} and~\ref{dousse2} in the same way that Theorem~\ref{th:corlov} generalises Andrews' identities and to unify all the above-mentioned generalisations of Schur's theorem. We prove the following.

\begin{theorem}
\label{refinement}
Let $\overline{D}(\ell_1, \dots, \ell_r;k,n)$ denote the number of overpartitions of $n$ into non-negative parts coloured $u_1, \dots, u_{r-1}$ or $u_r$, having $\ell_i$ parts coloured $u_i$ for all $i \in \{1, \dots,r\}$ and $k$ non-overlined parts. Let $\overline{E}(\ell_1, \dots, \ell_r;k,n)$ denote the number of overpartitions $\lambda_1+\cdots+ \lambda_s$ of $n$ into non-negative parts coloured $\tilde{u}_1, \dots, \tilde{u}_{2^r-2}$ or $\tilde{u}_{2^r-1}$, such that for all $i \in \{1, \dots,r\}$, $\ell_i$ parts have $u_i$ as one of their primary colours, having $k$ non-overlined parts and satisfying the difference conditions
$$\lambda_i - \lambda_{i+1} \geq w(c(\lambda_{i+1}))+\chi(\overline{\lambda_{i+1}})-1+\delta(c(\lambda_i),c(\lambda_{i+1})),$$
where $\chi(\overline{\lambda_{i+1}}) =1$ if $\lambda_{i+1}$ is overlined and $0$ otherwise.

Then for all $\ell_1, \dots, \ell_r, k, n \geq 0$,
$$\overline{D}(\ell_1, \dots, \ell_r;k,n)= \overline{E}(\ell_1, \dots, \ell_r;k,n).$$
\end{theorem}
The proof of Theorem~\ref{refinement} relies on the combination of the method of weighted words of Alladi and Gordon~\cite{Alladi} and the $q$-difference equations techniques introduced by the author in~\cite{Doussegene}. This idea of mixing the method of weighted words with $q$-difference equations was first introduced by the author in a recent paper~\cite{Doussesilweight} to prove a refinement and companion of Siladi\'c's theorem~\cite{Siladic}, a partition identity that first arose in the study of Lie algebras.

As in the work of Corteel and Lovejoy, we can take advantage of the symmetries in Theorem~\ref{refinement}.
Let $\sigma \in S_r$ be a permutation. For every overpartition $\lambda$ counted by $\overline{E}(\ell_1, \dots, \ell_r;k,n)$, we define a new overpartition $\lambda^{\sigma}$ obtained by setting $\lambda_i^{\sigma} = \lambda_i$ and $c(\lambda_i^{\sigma})=\sigma(c(\lambda_i))$, and overlining $\lambda_i^{\sigma}$ if and only if $\lambda_i$ was overlined. This mapping is reversible and doesn't change $w(c(\lambda_{i+1}))$ or $\chi(\overline{\lambda_{i+1}})$, so the difference condition we obtain on $\lambda^{\sigma}$ is
\begin{equation}
\label{cond_diff_sigma_ov}
\lambda_i^{\sigma} - \lambda_{i+1}^{\sigma} \geq w(c(\lambda_{i+1}^{\sigma}))+\chi(\overline{\lambda_{i+1}^{\sigma}})-1+\delta(\sigma^{-1}(c(\lambda_i^{\sigma})),\sigma^{-1}(c(\lambda_{i+1}^{\sigma}))).
\end{equation}
Thus $\overline{E}(\ell_1, \dots, \ell_r;k,n) = \overline{E}^{\sigma}(\ell_{\sigma^{-1}(1)}, \dots, \ell_{\sigma^{-1}(r)};k,n)$, where $\overline{E}^{\sigma}(\ell_1, \dots, \ell_r;k,n)$ denotes the number of overpartitions of $n$ into non-negative parts, each part being coloured in one of the colours $\tilde{u}_1, \dots, \tilde{u}_{2^r-1}$, such that for all  $i \in \{1, \dots,r\}$, $\ell_i$ parts have $u_i$ as one of their primary colours, satisfying the difference condition~\eqref{cond_diff_sigma_ov}.

Moreover, by doing the same transformation on the overpartitions counted by $D(\ell_1, \dots, \ell_r;n),$ one can see that
$$\overline{D}(\ell_1, \dots, \ell_r;k,n)=\overline{D}(\ell_{\sigma^{-1}(1)}, \dots, \ell_{\sigma^{-1}(r)};k,n).$$

Thus one has
$$\overline{D}(\ell_{\sigma^{-1}(1)}, \dots, \ell_{\sigma^{-1}(r)};k,n)=\overline{E}^{\sigma}(\ell_{\sigma^{-1}(1)}, \dots, \ell_{\sigma^{-1}(r)};k,n),$$
and relabelling the colours gives

\begin{corollary}
\label{cor}
For every permutation $\sigma\in S_r$,
$$\overline{D}(\ell_{1}, \dots, \ell_{r};k,n)=\overline{E}^{\sigma}(\ell_{1}, \dots, \ell_{r};k,n).$$
\end{corollary}

We introduce one more notation. For $\alpha, \beta \in A'$, let 
$$\delta_A(\alpha, \beta) :=  \begin{cases}
1 \text{ if } z_A(\alpha) < v_A(\beta)\\
0 \text{ otherwise.}
\end{cases}$$
We also extend the permutations to every integer $\alpha = a(i_1) + \cdots + a(i_s) \in A'$ by setting
$$\sigma(\alpha)= a(\sigma(i_1))+ \cdots + a(\sigma(i_s)).$$

By doing the transformations
$$ q \rightarrow q^N, u_1 \rightarrow u_1q^{a(1)}, \dots u_r \rightarrow u_r q^{a(r)},$$
we obtain the following generalisation and refinement of Theorem~\ref{dousse}. Details are given in Section 2.

\begin{theorem}
\label{refdousse}
Let $\overline{D}(A_N;\ell_1, \dots, \ell_r;k,n)$ denote the number of overpartitions of $n$ into parts taken from $A_N$, having $k$ non-overlined parts, such that for all $i \in \{1, \dots, r\}$, $\ell_i$ parts are congruent to $a(i)$ modulo $N$. Let $\overline{E}^{\sigma}(A'_N;\ell_1, \dots, \ell_r;k,n)$ denote the number of overpartitions of $n$ into parts taken from $A'_N$ of the form $n=\lambda_1+\cdots+ \lambda_s$, having $k$ non-overlined parts, such that for all $i \in \{1, \dots, r\}$, $\ell_i$ is the number of parts $\lambda_j$ such that $\beta_N(\lambda_j)$ uses $a(i)$ in its defining sum, and satisfying the difference conditions
\begin{align*}
\lambda_i - \lambda_{i+1} \geq &N \left(w_A\left(\beta_N(\lambda_{i+1})\right) -1 +\chi(\ov{\lambda_{i+1}}) + \delta_A(\sigma(\beta_N(\lambda_i)),\sigma(\beta_N(\lambda_{i+1}))) \right)\\
&+\beta_N(\lambda_{i})-\beta_N(\lambda_{i+1}),
\end{align*}
where $\chi(\ov{\lambda_{i+1}})=1$ if $\lambda_{i+1}$ is overlined and $0$ otherwise.

Then for all $\ell_1, \dots, \ell_r,k,n \geq 0$, $\sigma \in S_r$,
$$\overline{D}(A_N;\ell_1, \dots, \ell_r;k,n) = \overline{E}^{\sigma}(A'_N;\ell_1, \dots, \ell_r;k,n).$$
\end{theorem}

Similarly, by using the transformations 
$$ q \rightarrow q^N, u_1 \rightarrow u_1q^{N-a(1)}, \dots u_r \rightarrow u_r q^{N-a(r)},$$
we obtain a refinement and generalisation of Theorem~\ref{dousse2} 

\begin{theorem}
\label{refdousse2}
Let $\overline{F}(-A_N;\ell_1, \dots, \ell_r;k,n)$ denote the number of overpartitions of $n$ into parts taken from $-A_N$, having $k$ non-overlined parts, such that for all $i \in \{1, \dots, r\}$, $\ell_i$ parts are congruent to $-a(i)$ modulo $N$. Let $\overline{G}^{\sigma}(-A'_N;\ell_1, \dots, \ell_r;k,n)$ denote the number of overpartitions of $n$ into parts taken from $-A'_N$ of the form $n=\lambda_1+\cdots+ \lambda_s$, having $k$ non-overlined parts, such that for all $i \in \{1, \dots, r\}$, $\ell_i$ is the number of parts $\lambda_j$ such that $\beta_N(-\lambda_j)$ uses $a(i)$ in its defining sum, and satisfying the difference conditions
\begin{align*}
\lambda_i - \lambda_{i+1} \geq &N \left(w_A\left(\beta_N(-\lambda_{i})\right) -1 +\chi(\ov{\lambda_{i+1}}) + \delta_A(\sigma(\beta_N(-\lambda_i)),\sigma(\beta_N(-\lambda_{i+1}))) \right)\\
&+\beta_N(-\lambda_{i+1})-\beta_N(-\lambda_{i}),
\end{align*}
and
$$
\lambda_s \geq N w_A(\beta_N(-\lambda_s))- \beta_N(-\lambda_s).
$$

Then for all $\ell_1, \dots, \ell_r,k,n \geq 0$, $\sigma \in S_r$,
$$\overline{F}(-A_N;\ell_1, \dots, \ell_r;k,n) = \overline{G}^{\sigma}(-A'_N;\ell_1, \dots, \ell_r;k,n).$$
\end{theorem}
Setting $k=0$ in Theorems~\ref{refdousse} and~\ref{refdousse2} recovers two theorems of Corteel and Lovejoy~\cite{Corteel}.

Theorem~\ref{dousse} corresponds to the case $\sigma= Id$ in Theorem~\ref{refdousse} and Theorem~\ref{dousse2} to the case $\sigma=(n,n-1, \dots,1)$ in Theorem~\ref{refdousse2}. Details on how to recover Theorems~\ref{refdousse} and~\ref{refdousse2} are also given in Section 2.

Theorem~\ref{refdousse} (resp. \ref{refdousse2}) gives $r! -1$ new companions to Theorem~\ref{dousse} (resp. \ref{dousse2}). For $r \geq 3,$ the companions of Theorem~\ref{dousse} are different from those of Theorem~\ref{dousse2}. For $r=2$, the two companions are the same when $a(1) =N-a(2)$.
For example, when $r=2, a(1)=1, a(2)=2$, setting $\sigma=(2,1)$ in Theorem~\ref{refdousse} gives the following theorem.

\begin{corollary}
\label{gene1N}
Let $\overline{D}(N,\ell_1,\ell_2;k,n)$ denote the number of overpartitions of $n$ into parts $\equiv 1,2 \mod N$ with $\ell_1$ parts $\equiv 1 \mod N$ and $\ell_2$ parts $\equiv 2 \mod N$ and having $k$ non-overlined parts.

Let $\overline{E}(N,\ell_1,\ell_2;k,n)$ denote the number of overpartitions $\lambda_1 + \cdots + \lambda_s$ of $n$ into parts $\equiv 1,2,3 \mod N$ with $\ell_1$ parts $\equiv 1,3 \mod N$ and $\ell_2$ parts $\equiv 2,3 \mod N$, having $k$ non-overlined parts, such that the entry $(x,y)$ in the matrix $M_N$ gives the minimal difference between $\lambda_i \equiv x \mod N$ and $\lambda_{i+1} \equiv y \mod N$:
$$
M_N=\bordermatrix{\text{} & 1&2&3 \cr 1 & N\chi(\overline{\lambda_{i+1}}) & N\chi(\overline{\lambda_{i+1}})-1 & N \left(\chi(\overline{\lambda_{i+1}})+1\right)-2 \cr 2 &N\left(\chi(\overline{\lambda_{i+1}})+1\right)+1 & N\chi(\overline{\lambda_{i+1}}) & N\left(\chi(\overline{\lambda_{i+1}})+1\right)-1 \cr 3 &N\chi(\overline{\lambda_{i+1}})+2 & N\chi(\overline{\lambda_{i+1}})+1 & N\left(\chi(\overline{\lambda_{i+1}})+1\right) }.
$$

Then $\overline{D}(N,\ell_1,\ell_2;k,n)=\overline{E}(N,\ell_1,\ell_2;k,n).$
\end{corollary}

On the other hand, setting $\sigma=Id$ in Theorem~\ref{refdousse2} gives the following.
\begin{corollary}
\label{gene2N}
Let $\overline{F}(N,\ell_1,\ell_2;k,n)$ denote the number of overpartitions of $n$ into parts $\equiv -1,-2 \mod N$ with $\ell_1$ parts $\equiv -1 \mod N$ and $\ell_2$ parts $\equiv -2 \mod N$ and having $k$ non-overlined parts.

Let $\overline{G}(N,\ell_1,\ell_2;k,n)$ denote the number of overpartitions $\lambda_1 + \cdots + \lambda_s$ of $n$ into parts $\equiv -1,-2,-3 \mod N$ with $\ell_1$ parts $\equiv -1,-3 \mod N$ and $\ell_2$ parts $\equiv -2,-3 \mod N$, having $k$ non-overlined parts, such that the entry $(x,y)$ in the matrix $M'_N$ gives the minimal difference between $\lambda_i \equiv x \mod N$ and $\lambda_{i+1} \equiv y \mod N$:
$$
M'_N=\bordermatrix{\text{} & -1&-2&-3 \cr -1 & N\chi(\overline{\lambda_{i+1}}) & N\left(\chi(\overline{\lambda_{i+1}})+1\right)+1 & N\chi(\overline{\lambda_{i+1}})+2 \cr -2 &N\chi(\overline{\lambda_{i+1}})-1 & N\chi(\overline{\lambda_{i+1}}) & N\chi(\overline{\lambda_{i+1}})+1 \cr -3 &N\left(\chi(\overline{\lambda_{i+1}})+1\right)-2 & N\left(\chi(\overline{\lambda_{i+1}})+1\right)-1 & N\left(\chi(\overline{\lambda_{i+1}})+1\right) }.
$$

Then $\overline{F}(N,\ell_1,\ell_2;k,n)=\overline{G}(N,\ell_1,\ell_2;k,n).$
\end{corollary}

When $N=3$, Corollaries~\ref{gene1N} and~\ref{gene2N} become the same companion to Lovejoy's theorem, where the difference conditions can be summarised as follows.
\begin{corollary}
\label{compschur}
Let $\overline{A}(\ell_1,\ell_2;k,n)$ denote the number of overpartitions of $n$ into parts $\equiv 1,2 \mod 3$ with $\ell_1$ parts $\equiv 1 \mod 3$ and $\ell_2$ parts $\equiv 2 \mod 3$ and having $k$ non-overlined parts.

Let $\overline{C}(\ell_1,\ell_2;k,n)$ denote the number of overpartitions $\lambda_1 + \cdots + \lambda_s$ of $n$ with $\ell_1$ parts $\equiv 1,3 \mod 3$ and $\ell_2$ parts $\equiv 2,3 \mod 3$, having $k$ non-overlined parts, satisfying the difference conditions
$$\lambda_i - \lambda_{i+1} \geq \begin{cases}
3 \chi(\overline{\lambda_{i+1}}) &\text{ if } \lambda_{i+1} \equiv 1 \mod 3 \text{ and } \lambda_i \equiv 1,3 \mod 3,\\
3\chi(\overline{\lambda_{i+1}})+4 &\text{ if } \lambda_{i+1} \equiv 1 \mod 3 \text{ and } \lambda_i \equiv 2 \mod 3,\\
3 \chi(\overline{\lambda_{i+1}})-1 &\text{ if } \lambda_{i+1} \equiv 2 \mod 3,\\
3 \chi(\overline{\lambda_{i+1}})+1 &\text{ if } \lambda_{i+1} \equiv 2 \mod 3.
\end{cases}
$$
Then $\overline{A}(\ell_1,\ell_2;k,n)=\overline{C}(\ell_1,\ell_2;k,n).$
\end{corollary}

The generalisations of Schur's theorem are summarised in Figure~\ref{fig:gene}, where $A \longrightarrow B$ means that the theorem corresponding to the infinite product $A$ is generalised by the theorem corresponding to the infinite product $B$. Here we use the classical notation
$$(a;q)_{n} := \prod_{j=0}^{n-1} (1-aq^j),$$
for $n \in \N \cup \{ \infty \}.$

%
%

\begin{figure}
\label{fig:gene}
\caption{The generalisations of Schur's theorem}
\begin{center}
\includegraphics[width=1.0\textwidth]{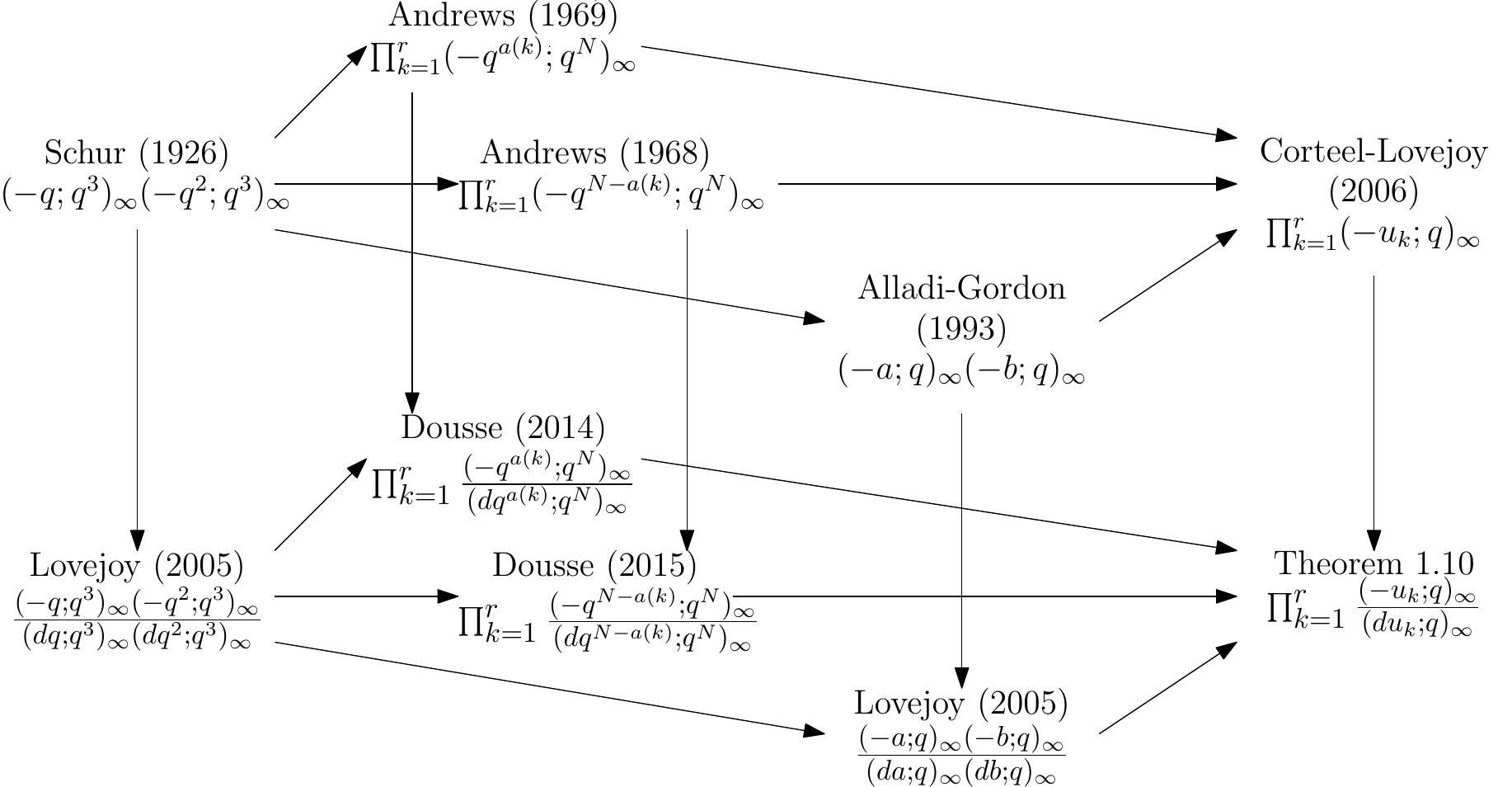}
\end{center}
\end{figure}

The rest of this paper organised as follows. In Section 2, we deduce Theorems~\ref{refdousse} and~\ref{refdousse2} from Theorem~\ref{refinement} and explain how refinements of Andrews' theorems for overpartitions (Theorems~\ref{dousse} and~\ref{dousse2}) can be derived from them. In Section 3, we prove Theorem~\ref{refinement} using the method of weighted words, $q$-difference equations and an induction. 

\section{Generalisations and refinements of Andrews' theorems for overpartitions}
We start by showing how to deduce Theorems~\ref{refdousse} and~\ref{refdousse2} from Theorem~\ref{refinement} and Corollary~\ref{cor}.

\subsection{Proof of Theorem~\ref{refdousse}}
Fix a permutation $\sigma \in S_r.$ By Corollary~\ref{cor}, we have
$$\overline{D}(\ell_1, \dots, \ell_r;k,n)=\overline{E}^{\sigma^{-1}}(\ell_{1}, \dots, \ell_{r};k,n).$$
Now transform the overpartitions counted by $\overline{D}(\ell_1, \dots, \ell_r;k,n)$ and $\overline{E}^{\sigma^{-1}}(\ell_{1}, \dots, \ell_{r};k,n)$ by transforming each part $\lambda_i$ of colour $\tilde{u}_j$ into a part $$ \lambda_i^{dil} = N \lambda_i + \alpha(j) = N \lambda_i + \epsilon_1(j) a(1) + \cdots + \epsilon_r(j) a(r).$$
This corresponds to doing the dilation $q \rightarrow q^N$ and the translations $u_i \rightarrow u_i q^{a(i)}$ for all $i \in \{1, \dots, r \}$ in the generating functions.
The number $k$ of non-overlined parts stays the same and the number $n$ partitioned becomes $$Nn + \ell_1 a(1) + \cdots + \ell_r a(r),$$ for both the overpartitions counted by $\overline{D}(\ell_1, \dots, \ell_r;k,n)$ and by $\overline{E}^{\sigma^{-1}}(\ell_1, \dots, \ell_r;k,n)$.

The parts before transformation were non-negative. After transformation, the parts of the overpartitions counted by $\overline{D}(\ell_1, \dots, \ell_r;k,n)$ belong to $A_N$ and those of the overpartitions counted by $\overline{E}^{\sigma^{-1}}(\ell_1, \dots, \ell_r;k,n)$ belong to $A'_N$.

Let us now turn to the difference conditions. Before transformation, the overpartitions counted by $\overline{E}^{\sigma^{-1}}(\ell_1, \dots, \ell_r;k,n)$ satisfied
$$\lambda_i - \lambda_{i+1} \geq w(c(\lambda_{i+1}))+\chi(\overline{\lambda_{i+1}})-1+\delta(\sigma(c(\lambda_i)),\sigma(c(\lambda_{i+1}))).$$
After the transformations, it becomes
\begin{align*}
\lambda_i^{dil} - \lambda_{i+1}^{dil} \geq &N \left(w_A\left(\alpha(c(\lambda_{i+1})\right)) -1 +\chi(\ov{\lambda_{i+1}^{dil}}) + \delta_A(\sigma(\alpha(c(\lambda_{i})),\sigma(\alpha(c(\lambda_{i+1})))) \right)\\
&+\alpha(c(\lambda_{i}))-\alpha(c(\lambda_{i+1})),
\end{align*}
By the definition of $\beta_N$ and the transformations, we have the equality $\alpha(c(\lambda_{i})) = \beta_N(\lambda_i^{dil}).$ Thus the difference condition becomes

\begin{align*}
\lambda_i^{dil} - \lambda_{i+1}^{dil} \geq &N \left(w_A\left(\beta_N(\lambda_{i+1}^{dil})\right) -1 +\chi(\ov{\lambda_{i+1}^{dil}}) + \delta_A(\sigma(\beta_N(\lambda_i^{dil})),\sigma(\beta_N(\lambda_{i+1}^{dil}))) \right)\\
&+\beta_N(\lambda_{i}^{dil})-\beta_N(\lambda_{i+1}^{dil}),
\end{align*}

This is exactly the difference condition from Theorem~\ref{refdousse}. This completes the proof.

\subsection{Proof of Theorem~\ref{refdousse2}}
Let us now turn to the proof of Theorem~\ref{refdousse2}.
As before, we have
$$\overline{D}(\ell_1, \dots, \ell_r;k,n)=\overline{E}^{\sigma^{-1}}(\ell_{1}, \dots, \ell_{r};k,n).$$
Now transform the overpartitions counted by $\overline{D}(\ell_1, \dots, \ell_r;k,n)$ and $\overline{E}^{\sigma^{-1}}(\ell_{1}, \dots, \ell_{r};k,n)$ by transforming each part $\lambda_i$ of colour $\tilde{u}_j$ into a part $$ \lambda_i^{dil'} = N \left(w(c(\lambda_i)+\lambda_i \right) - \alpha(j) = N \left(w(c(\lambda_i)+\lambda_i \right) - \epsilon_1(j) a(1) - \cdots - \epsilon_r(j) a(r).$$
This corresponds to doing the dilation $q \rightarrow q^N$ and the translations $u_i \rightarrow u_i q^{N-a(i)}$ for all $i \in \{1, \dots, r \}$ in the generating functions.
The number $k$ of non-overlined parts stays the same and the number $n$ partitioned becomes $$N (n + \ell_1 + \cdots + \ell_r) - \ell_1 a(1) - \cdots - \ell_r a(r),$$ for both the overpartitions counted by $\overline{D}(\ell_1, \dots, \ell_r;k,n)$ and by $\overline{E}^{\sigma^{-1}}(\ell_1, \dots, \ell_r;k,n)$.

The parts before transformation were non-negative. After transformation, the parts of the overpartitions counted by $\overline{D}(\ell_1, \dots, \ell_r;k,n)$ belong to $-A_N$ and those of the overpartitions counted by $\overline{E}^{\sigma^{-1}}(\ell_1, \dots, \ell_r;k,n)$ belong to $A'_N$, with the additional condition that for all $i$,
\begin{equation}
\label{minimum}
\lambda_i^{dil'} \geq N w(c(\lambda_i)) - \alpha(c(\lambda_i)) = N w_A(\beta_N(-\lambda_i^{dil'})) - \beta_N(-\lambda_i^{dil'})).
\end{equation}
Indeed by the definition of $\beta_N$ and the transformations, we have now $\alpha(c(\lambda_{i})) = \beta_N(-\lambda_i^{dil'}).$

The difference condition for the overpartitions counted by $\overline{E}^{\sigma^{-1}}(\ell_1, \dots, \ell_r;k,n)$ was
$$\lambda_i - \lambda_{i+1} \geq w(c(\lambda_{i+1}))+\chi(\overline{\lambda_{i+1}})-1+\delta(\sigma(c(\lambda_i)),\sigma(c(\lambda_{i+1}))).$$
After the transformations, it becomes
\begin{align*}
\lambda_i^{dil'} - \lambda_{i+1}^{dil'} \geq &N \left(w_A\left(\alpha(c(\lambda_{i})\right)) -1 +\chi(\ov{\lambda_{i+1}^{dil}}) + \delta_A(\sigma(\alpha(c(\lambda_{i})),\sigma(\alpha(c(\lambda_{i+1})))) \right)\\
&-\alpha(c(\lambda_{i}))+\alpha(c(\lambda_{i+1})),
\end{align*}
which is equivalent to
\begin{align*}
\lambda_i^{dil'} - \lambda_{i+1}^{dil'} \geq &N \left(w_A\left(\beta_N(-\lambda_{i}^{dil})\right) -1 +\chi(\ov{\lambda_{i+1}^{dil'}}) + \delta_A(\sigma(\beta_N(-\lambda_i^{dil})),\sigma(\beta_N(-\lambda_{i+1}^{dil}))) \right)\\
&-\beta_N(-\lambda_{i}^{dil'})+\beta_N(-\lambda_{i+1}^{dil'}),
\end{align*}

This is exactly the difference condition from Theorem~\ref{refdousse2}. This completes the proof.

\subsection{Refinement of Theorem~\ref{dousse}}
We now want to show that the case $\sigma=Id$ in Theorem~\ref{refdousse} is actually a refinement of Theorem~\ref{dousse}. To do so, let us reformulate Theorem~\ref{dousse}. The minimal difference between two consecutive parts $\lambda_i$ and $\lambda_{i+1}$ is 
$$\lambda_i - \lambda_{i+1} \geq N \left(w_A\left(\beta_N(\lambda_{i+1})\right) -1 +\chi(\ov{\lambda_{i+1}}) \right)+v_A(\beta_N(\lambda_{i+1}))-\beta_N(\lambda_{i+1}).$$
But by the definition of $\beta_N$, $\lambda_i - \lambda_{i+1}$ is always congruent to $\beta_N(\lambda_{i})-\beta_N(\lambda_{i+1})$ modulo $N$. Therefore the difference condition is actually equivalent to having a minimal difference
$$N\left(w_A\left(\beta_N(\lambda_{i+1})\right) -1 +\chi(\ov{\lambda_{i+1}}) \right)+\beta_N(\lambda_{i})-\beta_N(\lambda_{i+1}),$$
if $v_A(\beta_N(\lambda_{i+1})) \leq \beta_N(\lambda_{i})$,
and
$$N\left(w_A\left(\beta_N(\lambda_{i+1})\right) +\chi(\ov{\lambda_{i+1}}) \right)+\beta_N(\lambda_{i})-\beta_N(\lambda_{i+1}),$$
if $v_A(\beta_N(\lambda_{i+1})) > \beta_N(\lambda_{i}).$

We will be able to conclude using the following lemma.
\begin{lemma}
\label{comparison}
For $\alpha,\beta \in A'$, we have $v_A(\alpha) > \beta$ if and only if $v_A(\alpha) > z_A(\beta)$.
\end{lemma}
\begin{proof}
By the definition of $z_A$, $z_A(\beta) \leq \beta$. Thus if $v_A(\alpha) > \beta$, then $v_A(\alpha) > z_A(\beta)$.

Let us now show the other implication. Assume that $v_A(\alpha) > z_A(\beta)$. 
If we write $z_A(\beta)=a(k)$, then $v_A(\alpha) \geq a(k+1)$, but by the definition of $A$, we know that for all $k$,
$$\sum_{i=1}^{k} a(i) <a(k+1).$$
Thus
$$v_A(\alpha) \geq a(k+1) > \sum_{i=1}^{k} a(i) \geq z_A(\beta).$$

\end{proof}

Hence by Lemma~\ref{comparison}, the difference condition in Theorem~\ref{dousse} is actually equivalent to
\begin{align*}
\lambda_i - \lambda_{i+1} \geq& N \left(w_A\left(\beta_N(\lambda_{i+1})\right)-1 +\chi(\ov{\lambda_{i+1}}) + \delta_A(\beta_N(\lambda_i),\beta_N(\lambda_{i+1}) \right)\\
&+\beta_N(\lambda_{i})-\beta_N(\lambda_{i+1}),
\end{align*}

which is exactly the difference condition of Theorem~\ref{refdousse} with $\sigma= Id$.

\subsection{Refinement of Theorem~\ref{dousse2}}
Finally, let us show that the case $\sigma= (n, n-1, \dots, 1)$ in Theorem~\ref{refdousse2} is actually a refinement of Theorem~\ref{dousse2}. To do so, let us reformulate Theorem~\ref{dousse2}. The minimal difference between two consecutive parts $\lambda_i$ and $\lambda_{i+1}$ is 
$$
\lambda_i - \lambda_{i+1} \geq N \left(w_A\left(\beta_N(-\lambda_{i})\right) -1 +\chi(\ov{\lambda_{i+1}})\right)+v_A(\beta_N(-\lambda_{i}))-\beta_N(-\lambda_{i}).
$$
But $\lambda_i - \lambda_{i+1}$ is always congruent to $-\beta_N(-\lambda_{i})+\beta_N(-\lambda_{i+1})$ modulo $N$. Therefore the difference condition is actually equivalent to having a minimal difference
$$N \left(w_A\left(\beta_N(-\lambda_{i})\right) -1 +\chi(\ov{\lambda_{i+1}})\right)-\beta_N(-\lambda_{i})+\beta_N(-\lambda_{i+1}),$$
if $v_A(\beta_N(-\lambda_{i})) \leq \beta_N(-\lambda_{i+1})$,
and
$$N \left(w_A\left(\beta_N(-\lambda_{i})\right) +\chi(\ov{\lambda_{i+1}})\right)-\beta_N(-\lambda_{i})+\beta_N(-\lambda_{i+1}),$$
if $v_A(\beta_N(-\lambda_{i})) > \beta_N(-\lambda_{i+1}).$

Again, by Lemma~\ref{comparison}, this difference condition is equivalent to
\begin{align*}
\lambda_i - \lambda_{i+1} \geq &N \left(w_A\left(\beta_N(-\lambda_{i})\right) -1 +\chi(\ov{\lambda_{i+1}}) + \delta_A(\beta_N(-\lambda_{i+1}),\beta_N(-\lambda_{i})) \right)\\
&+\beta_N(-\lambda_{i+1})-\beta_N(-\lambda_{i}).
\end{align*}
But when $\sigma= (n, n-1, \dots, 1)$, then $$\delta_A(\sigma(\beta_N(-\lambda_i)),\sigma(\beta_N(-\lambda_{i+1})))=\delta_A(\beta_N(-\lambda_{i+1}),\beta_N(-\lambda_{i})),$$
so we obtain exactly the same difference condition in Theorem~\ref{dousse2}.

Finally, as $\lambda_s$ is always congruent to $-\beta_N(-\lambda_s)$ modulo $N$, the condition $\lambda_s \geq N w_A(\beta_N(-\lambda_s))- \beta_N(-\lambda_s)$ is equivalent to $\lambda_s \geq N (w_A(\beta_N(-\lambda_s))- 1).$
This completes the proof.

\section{Proof of Theorem~\ref{refinement}}
Let us now turn to the proof of Theorem~\ref{refinement}.

It is clear that the generating function for the overpartitions with congruence conditions is
$$ \sum_{\ell_1, \dots, \ell_r,k,n \geq 0} D(\ell_1, \dots, \ell_r;k,n) u_1^{\ell_1} \cdots u_r^{\ell_r} d^k q^n = \prod_{k=1}^r \frac{(-u_k;q)_{\infty}}{(du_k;q)_{\infty}}.$$
The difficult task is to show that the generating function for overpartitions enumerated by $E(\ell_1, \dots, \ell_r;k,n)$ is the same. To do so, we adapt techniques introduced in~\cite{Doussegene} by taking colours into account. First, we establish the $q$-difference equation satisfied by the generating function with one added variable counting the number of parts, and then we prove by induction that a function satisfying this $q$-difference equation is equal to $\prod_{k=1}^r \frac{(-u_k;q)_{\infty}}{(du_k;q)_{\infty}}$ when the added variable is equal to $1$.

\subsection{The $q$-difference equation}
Let us first establish the $q$-difference equation.
Let $p_{i_{\tilde{u}_j}}(\ell_1, \dots, \ell_r;k,m,n)$ denote the number of overpartitions counted by $E(\ell_1, \dots, \ell_r;k,n)$ having $m$ parts such that the smallest part is at least $i_{\tilde{u}_j}$ (the non-negative integers are ordered according to their colours : $0_{\tilde{u}_1} < \cdots < 0_{\tilde{u}_{2^r-1}} < 1_{\tilde{u}_1} < \cdots$).

We first prove the following lemma.

\begin{lemma}
\label{lemma1}
If $1 \leq j \leq 2^r-2$, then
\begin{equation}
\label{eq1}
\begin{aligned}
&p_{0_{\tilde{u}_j}}(\ell_1, \dots, \ell_r;k,m,n) - p_{0_{\tilde{u}_{j+1}}}(\ell_1, \dots, \ell_r;k,m,n)
\\ &= p_{0_{v(\tilde{u}_j)}}(\ell_1 -\epsilon_1(j), \dots, \ell_r-\epsilon_r(j);k,m-1,n-(m-1)w(\tilde{u}_j))
\\&+ p_{0_{v(\tilde{u}_j)}}(\ell_1 -\epsilon_1(j), \dots, \ell_r-\epsilon_r(j);k-1,m-1,n-(m-1)(w(\tilde{u}_j)-1)),
\end{aligned}
\end{equation}

\begin{equation}
\label{eq1bis}
\begin{aligned}
&p_{0_{\tilde{u}_{2^r-1}}}(\ell_1, \dots, \ell_r;k,m,n) - p_{1_{\tilde{u}_{1}}}(\ell_1, \dots, \ell_r;k,m,n)
\\ &= p_{0_{u_1}}(\ell_1 -1, \dots, \ell_r-1;k,m-1,n-(m-1)r)
\\&+ p_{0_{u_1}}(\ell_1 -1, \dots, \ell_r-1;k-1,m-1,n-(m-1)(r-1)),
\end{aligned}
\end{equation}

\begin{equation}
\label{eq2}
p_{1_{\tilde{u}_1}}(\ell_1, \dots, \ell_r;k,m,n) = p_{0_{\tilde{u}_1}}(\ell_1, \dots, \ell_r;k,m,n-m).
\end{equation}
\end{lemma}

\begin{proof}
Let us first prove~\eqref{eq1}.
The quantity $$p_{0_{\tilde{u}_j}}(\ell_1, \dots, \ell_r;k,m,n) - p_{0_{\tilde{u}_{j+1}}}(\ell_1, \dots, \ell_r;k,m,n)$$ is the number of overpartitions $\lambda_1 +\cdots+ \lambda_m$ of $n$ enumerated by $p_{0_{\tilde{u}_j}}(\ell_1, \dots, \ell_r;k,m,n)$ such that the smallest part is equal to $0_{\tilde{u}_j}$.

If $\lambda_m = \ov{0_{\tilde{u}_j}}$ is overlined, then by the difference conditions in Theorem~\ref{refinement},
$$\lambda_{m-1} \geq 1 + w(\tilde{u}_j)+\delta(c(\lambda_{m-1}),v(\tilde{u}_j)).$$
This is equivalent to
$$\lambda_{m-1} \geq \begin{cases}  w(\tilde{u}_j) \text{ if $\lambda_{m-1}$'s colour is at least $v(\tilde{u}_j)$,}\\
1 + w(\tilde{u}_j) \text{ if $\lambda_{m-1}$'s colour is less than $v(\tilde{u}_j)$.}
\end{cases}$$
In other words,
$$\lambda_{m-1} \geq (w(\tilde{u}_j))_{v(\tilde{u}_j)}.$$
Then we remove $\lambda_m = \ov{0_{\tilde{u}_j}}$ and subtract $w(\tilde{u}_j)$ from every other part. For all $i \in \{1, \dots, r\},$ the number of parts using $u_i$ as a primary colour decreases by $1$ if and only if $u_i$ appeared in $\tilde{u}_j$, ie. if and only if $\epsilon_i(j)=1$. The number of parts is reduced to $m-1$, the number of non-overlined parts is still $k$, and the number partitioned is now $n-(m-1)w(\tilde{u}_j)$. Moreover the smallest part is now at least $0_{v(\tilde{u}_j)}.$ Therefore we obtain an overpartition counted by $$p_{0_{v(\tilde{u}_j)}}(\ell_1 -\epsilon_1(j), \dots, \ell_r-\epsilon_r(j);k,m-1,n-(m-1)w(\tilde{u}_j)).$$

If $\lambda_m = 0_{\tilde{u}_j}$ is not overlined, then in the same way as before, by the difference conditions in Theorem~\ref{refinement},
$$\lambda_{m-1} \geq w(\tilde{u}_j)-1 +\delta(c(\lambda_{m-1}),v(\tilde{u}_j)).$$
In other words,
$$\lambda_{m-1} \geq (w(\tilde{u}_j)-1)_{v(\tilde{u}_j)}.$$
Then we remove $\lambda_m = 0_{\tilde{u}_j}$ and subtract $w(\tilde{u}_j)-1$ from every other part. For all $i \in \{1, \dots, r\},$ the number of parts using $u_i$ as a primary colour decreases by $1$ if and only if $\epsilon_i(j)=1$. The number of parts is reduced to $m-1$, the number of non-overlined parts is reduced to $k-1$, and the number partitioned is now $n-(m-1)(w(\tilde{u}_j)-1)$. Moreover the smallest part is now at least $0_{v(\tilde{u}_j)}.$ Therefore we obtain an overpartition counted by $$p_{0_{v(\tilde{u}_j)}}(\ell_1 -\epsilon_1(j), \dots, \ell_r-\epsilon_r(j);k-1,m-1,n-(m-1)(w(\tilde{u}_j)-1)).$$

The proof of~\eqref{eq1bis} is exactly the same with $j=2^r-1$.

Finally, to prove~\eqref{eq2}, we take a partition enumerated by $p_{1_{\tilde{u}_1}}(\ell_1, \dots, \ell_r;k,m,n)$ and subtract $1$ from each part. We obtain a partition enumerated by $$p_{0_{\tilde{u}_1}}(\ell_1, \dots, \ell_r; k,m,n-m).$$
\end{proof}

These recurrences can be translated as $q$-difference equations on generating functions.

Let us define
\begin{equation}
\label{def_fi}
\begin{aligned}
f_{i_{\tilde{u}_j}}(x) &=f_{i_{\tilde{u}_j}}(u_1, \dots, u_r, d,x,q)\\
&:= \sum_{\ell_1, \dots, \ell_r,k,m,n \geq 0} p_{i_{\tilde{u}_j}}(\ell_1, \dots, \ell_r;k,m,n) u_1^{\ell_1} \cdots u_r^{\ell_r} d^k x^m q^n.
\end{aligned}
\end{equation}

We want to find an expression for $f_{0_{u_1}}(1)$, which is the generating function for all overpartitions counted by  $E(\ell_1, \dots, \ell_r;k,n)$.

Lemma~\ref{lemma1} implies the following equations.

\begin{lemma}
\label{lemma2}
If $1 \leq j \leq 2^r-2$, then
\begin{equation}
\label{eqf1}
\begin{aligned}
f_{0_{\tilde{u}_j}}(x) - f_{0_{\tilde{u}_{j+1}}}(x) &= x u_1^{\epsilon_1(j)} \cdots u_r^{\epsilon_r(j)} f_{0_{v(\tilde{u}_j)}}(xq^{w(\tilde{u}_j)})
\\&+ dx u_1^{\epsilon_1(j)} \cdots u_r^{\epsilon_r(j)} f_{0_{v(\tilde{u}_j)}}(xq^{w(\tilde{u}_j)-1}),
\end{aligned}
\end{equation}

\begin{equation}
\label{eqf1bis}
\begin{aligned}
f_{0_{\tilde{u}_{2^r-1}}}(x) - f_{1_{\tilde{u}_{1}}}(x) &= x u_1 \cdots u_r  f_{0_{u_1}}(xq^r)+ dx u_1 \cdots u_r f_{0_{u_1}}(xq^{r-1}),
\end{aligned}
\end{equation}

\begin{equation}
\label{eqf2}
f_{1_{\tilde{u}_1}}(x) = f_{0_{\tilde{u}_1}}(xq).
\end{equation}
\end{lemma}

Let $ 2 \leq k \leq r$. Note that $\tilde{u}_{2^{k-1}}= u_k.$ Adding equations~\eqref{eqf1} together for $1 \leq j \leq 2^{k-1}-1$ gives

\begin{equation}
\label{eq3.5}
\begin{aligned}
f_{0_{u_1}}(x) - f_{0_ {u_k}}(x) =\sum_{j=1}^{2^{k-1}-1} &\left( x u_1^{\epsilon_1(j)} \cdots u_r^{\epsilon_r(j)} f_{0_{v(\tilde{u}_j)}}(xq^{w(\tilde{u}_j)}) \right. \\
& \left. + dx u_1^{\epsilon_1(j)} \cdots u_r^{\epsilon_r(j)} f_{0_{v(\tilde{u}_j)}}(xq^{w(\tilde{u}_j)-1})\right).
\end{aligned}
\end{equation}
In the same way, adding equations~\eqref{eqf1} together for $2^{k-2} \leq j \leq 2^{k-1}-1$ gives

\begin{equation}
\label{eq3.6}
\begin{aligned}
f_{0_{u_{k-1}}}(x) - f_{0_{u_k}}(x) =\sum_{j=2^{k-2}}^{2^{k-1}-1} &\left( x u_1^{\epsilon_1(j)} \cdots u_r^{\epsilon_r(j)} f_{0_{v(\tilde{u}_j)}}(xq^{w(\tilde{u}_j)}) \right. \\
& \left. + dx u_1^{\epsilon_1(j)} \cdots u_r^{\epsilon_r(j)} f_{0_{v(\tilde{u}_j)}}(xq^{w(\tilde{u}_j)-1})\right).
\end{aligned}
\end{equation}
For all  $2^{k-2} \leq j \leq 2^{k-1}-1$, $\tilde{u}_j$ is of the form
$$\tilde{u}_j=u_1^{\epsilon_1(j)} \cdots u_{k-2}^{\epsilon_{k-2}(j)} u_{k-1}.$$
Thus~\eqref{eq3.6} can be rewritten as

\begin{align*}
f_{0_{u_{k-1}}}(x) - f_{0_ {u_k}}(x) &= xu_{k-1} f_{0_{u_{k-1}}}(xq) + dxu_{k-1} f_{0_{u_{k-1}}}(x)
\\&+ q^{-1}u_{k-1} \sum_{j=1}^{2^{k-2}-1} \left( xq u_1^{\epsilon_1(j)} \cdots u_{k-2}^{\epsilon_{k-2}(j)}f_{0_{v(\tilde{u}_j)}}(xq^{w(\tilde{u}_j)+1}) \right.\\
&\qquad \qquad \qquad \qquad \left.+ dxq u_1^{\epsilon_1(j)} \cdots u_{k-2}^{\epsilon_{k-2}(j)}f_{0_{v(\tilde{u}_j)}}(xq^{w(\tilde{u}_j)})\right)\\
&= xu_{k-1} f_{0_{u_{k-1}}}(xq) + dxu_{k-1} f_{0_{u_{k-1}}}(x)
\\&+ q^{-1}u_{k-1} \left( f_{0_{u_{1}}}(xq) - f_{0_ {u_{k-1}}}(xq) \right),
\end{align*}
where we used~\eqref{eq3.5} with $k$ replaced by $k-1$ and $x$ replaced by $xq$ to obtain the last equality.

Thus
\begin{equation}
\label{eq3.7}
\begin{aligned}
f_{0_ {u_k}}(x) &= (1-dxu_{k-1}) f_{0_{u_{k-1}}}(x) -q^{-1}u_{k-1}  f_{0_{u_{1}}}(xq) \\
&+ q^{-1}u_{k-1}  (1-xq) f_{0_{u_{k-1}}}(xq).
\end{aligned}
\end{equation}

In the same way, on can show that
\begin{equation}
\label{eq3.7bis}
\begin{aligned}
f_{1_ {u_1}}(x) &= (1-dxu_{r}) f_{0_{u_{r}}}(x) -q^{-1}u_{r}  f_{0_{u_{1}}}(xq) \\
&+ q^{-1}u_{r}  (1-xq) f_{0_{u_{r}}}(xq).
\end{aligned}
\end{equation}

We are almost ready to give the $q$-difference equation relating functions $f_{0_{u_1}}\left(xq^{k}\right)$ together for $k \geq 0$. To do so, recall that the $q$-binomial coefficients are defined as
$${m \brack r}_q :=
\begin{cases}
\frac{\left(1-q^m\right)\left(1-q^{m-1}\right) \dots \left(1-q^{m-r+1}\right)}{\left(1-q\right) \left(1-q^2\right) \dots \left(1-q^r\right)}\ \text{if}\ 0 \leq r \leq m,\\
0 \ \text{otherwise}.
\end{cases}$$
They are $q$-analogues of the binomial coefficients and satisfy $q$-analogues of the Pascal triangle identity~\cite{Gasper}.

\begin{proposition}
\label{pascal}
For all integers $0 \leq r \leq m$,
\begin{equation}
\label{pascal1}
{m \brack r}_q = q^r {m-1 \brack r}_q + {m-1 \brack r-1}_q,
\end{equation}
\begin{equation}
\label{pascal2}
{m \brack r}_q ={m-1 \brack r}_q + q^{m-r} {m-1 \brack r-1}_q.
\end{equation}
\end{proposition}

The following lemma will help us to obtain the desired $q$-difference equation.
\begin{lemma}
\label{conj}
For $1 \leq k \leq r$, we have
\begin{equation}
\label{eq}
\begin{aligned}
\prod_{i=1}^{k-1} &\left(1-dxu_i\right) f_{0_{u_1}}(x) = f_{0_{u_k}}(x) 
\\ + \sum_{i=1}^{k-1} &\left( \sum_{m=0}^{k-i-1} d^m \sum_{\substack{1 \leq j < 2^{k-1} \\ w(\tilde{u}_j)=i+m}} x u_1^{\epsilon_1(j)} \cdots u_{k-1}^{\epsilon_{k-1}(j)} \left( (-x)^{m-1} {i+m-1 \brack m-1}_{q} \right. \right. \\
& \left. \vphantom{\sum_{\substack{1 \leq j < 2^{k-1} \\ w(\tilde{u}_j)=i+m}}} \left. + (-x)^m {i+m \brack m}_{q} \right) \right) \times \prod_{h=1}^{i-1} \left(1-xq^{h}\right) f_{0_{u_1}}\left(xq^{i}\right).
\end{aligned}
\end{equation}
\end{lemma}
\begin{proof}
The proof relies on an induction on $k$.
For $k=1$,~\eqref{eq} reduces to the trivial equation $f_{0_{u_1}}(x) = f_{0_{u_1}}(x).$
Now assume that~\eqref{eq} is true for some $1 \leq k \leq r-1$ and show it is also true for $k+1$.
Let us define
\begin{align*}
s_k(x) : = \sum_{i=1}^{k-1} &\left( \sum_{m=0}^{k-i-1} d^m \sum_{\substack{1 \leq j < 2^{k-1} \\ w(\tilde{u}_j)=i+m}} x u_1^{\epsilon_1(j)} \cdots u_{k-1}^{\epsilon_{k-1}(j)} \left( (-x)^{m-1} {i+m-1 \brack m-1}_{q} \right. \right. \\
& \left. \vphantom{\sum_{\substack{1 \leq j < 2^{k-1} \\ w(\tilde{u}_j)=i+m}}} \left. + (-x)^m {i+m \brack m}_{q} \right) \right) \times \prod_{h=1}^{i-1} \left(1-xq^{h}\right) f_{0_{u_1}}\left(xq^{i}\right).
\end{align*}
We want to show that
$$\prod_{i=1}^{k} \left(1-dxu_i\right) f_{0_{u_1}}(x) = f_{0_{u_{k+1}}}(x)  + s_{k+1}(x).$$
One has
\begin{align*}
\prod_{i=1}^{k} &\left(1-dxu_i\right) f_{0_{u_1}}(x) - f_{0_{u_{k+1}}}(x)
\\=& \left(1-dxu_k\right) \left( \prod_{i=1}^{k-1} \left(1-dxu_i\right) f_{0_{u_1}}(x) - f_{0_{u_k}}(x) \right)
\\&+ \left(1-dxu_k\right) f_{0_{u_k}}(x) - f_{0_{u_{k+1}}}(x)
\\=&\left(1-dxu_k\right) s_k(x)+ q^{-1}u_k  f_{0_{u_1}}(xq) -q^{-1}u_k\left(1-xq \right) f_{0_{u_k}}(xq),
\end{align*}
where we used the induction hypothesis and~\eqref{eq3.7} in the last equality.
Then
\begin{align*}
\prod_{i=1}^{k} &\left(1-dxu_i\right) f_{0_{u_1}}(x) - f_{0_{u_{k+1}}}(x)
\\=&\left(1-dxu_k\right) s_k(x)+ q^{-1}u_k  f_{0_{u_1}}(xq)
\\&-q^{-1}u_k\left(1-xq \right) \left( \prod_{i=1}^{k-1} \left(1-dxqu_i\right) f_{0_{u_1}}(xq) -s_k\left(xq\right)\right)
\\=& \left(1-dxu_k\right) s_k(x)+q^{-1}u_k\left(1-xq \right) s_k\left(xq\right)
\\&+ q^{-1}u_k \left( 1- \left(1-xq \right) \prod_{i=1}^{k-1}\left(1-dxqu_i\right)\right) f_{0_{u_1}}(xq)
\\=& \left(1-dxu_k\right) s_k(x)+q^{-1}u_k\left(1-xq \right) s_k\left(xq\right)
\\&+ q^{-1}u_k \left( 1- \left(1-xq \right) \left( 1 + \sum_{m=1}^{k-1}\sum_{\substack{1 \leq j < 2^{k-1} \\ w(\tilde{u}_j)=m}} (-dxq)^m u_1^{\epsilon_1(j)} \cdots u_{k-1}^{\epsilon_{k-1}(j)}\right)\right) f_{0_{u_1}}(xq)
\\=&\left(1-dxu_k\right) s_k(x)+q^{-1}u_k\left(1-xq \right) s_k\left(xq\right)
\\&+ q^{-1}u_k \left( xq + \sum_{m=1}^{k-1} d^m \sum_{\substack{1 \leq j < 2^{k-1} \\ w(\tilde{u}_j)=m}} xq u_1^{\epsilon_1(j)} \cdots u_{k-1}^{\epsilon_{k-1}(j)} \left( (-xq)^{m-1} + (-xq)^m\right)\right) f_{0_{u_1}}(xq)
\\=&\left(1-dxu_k\right) s_k(x)+q^{-1}u_k\left(1-xq \right) s_k\left(xq\right)
\\&+\left( xu_k + \sum_{m=1}^{k-1} d^m \sum_{\substack{2^{k-1} < j < 2^{k} \\ w(\tilde{u}_j)=m+1}} x u_1^{\epsilon_1(j)} \cdots u_{k-1}^{\epsilon_{k-1}(j)} u_k \left( (-xq)^{m-1} + (-xq)^m\right)\right) f_{0_{u_1}}(xq).
\end{align*}
Expanding and replacing $s_k$ by its definition, we get
\begin{align*}
\prod_{i=1}^{k} &\left(1-dxu_i\right) f_{0_{u_1}}(x) - f_{0_{u_{k+1}}}(x)
\\=& \sum_{i=1}^{k-1} \left( \sum_{m=0}^{k-i-1} d^m \sum_{\substack{1 \leq j < 2^{k-1} \\ w(\tilde{u}_j)=i+m}} x u_1^{\epsilon_1(j)} \cdots u_{k-1}^{\epsilon_{k-1}(j)} \left( (-x)^{m-1} {i+m-1 \brack m-1}_{q} \right. \right.\\
&\qquad  \left. \vphantom{ \sum_{\substack{1 \leq j < 2^{k-1} \\ w(\tilde{u}_j)=i+m}}} \left.  + (-x)^m {i+m \brack m}_{q} \right) \right) \prod_{h=1}^{i-1} \left(1-xq^{h}\right)f_{0_{u_1}}(xq^i)
\\+& \sum_{i=1}^{k-1} \left( \sum_{m=0}^{k-i-1} d^{m+1} \sum_{\substack{1 \leq j < 2^{k-1} \\ w(\tilde{u}_j)=i+m}} x u_1^{\epsilon_1(j)} \cdots u_{k-1}^{\epsilon_{k-1}(j)} u_k \left( (-x)^{m} {i+m-1 \brack m-1}_{q} \right. \right.\\
&\qquad  \left. \vphantom{ \sum_{\substack{1 \leq j < 2^{k-1} \\ w(\tilde{u}_j)=i+m}}} \left.  + (-x)^{m+1} {i+m \brack m}_{q} \right) \right) \prod_{h=1}^{i-1} \left(1-xq^{h}\right)f_{0_{u_1}}(xq^i)
\\+& q^{-1}u_k\left(1-xq \right) \times
\\&\sum_{i=1}^{k-1} \left( \sum_{m=0}^{k-i-1} d^m \sum_{\substack{1 \leq j < 2^{k-1} \\ w(\tilde{u}_j)=i+m}} xq u_1^{\epsilon_1(j)} \cdots u_{k-1}^{\epsilon_{k-1}(j)} \left( (-xq)^{m-1} {i+m-1 \brack m-1}_{q} \right. \right.\\
&\qquad  \left. \vphantom{ \sum_{\substack{1 \leq j < 2^{k-1} \\ w(\tilde{u}_j)=i+m}}} \left.  + (-xq)^m {i+m \brack m}_{q} \right) \right) \prod_{h=1}^{i-1} \left(1-xq^{h+1}\right)f_{0_{u_1}}(xq^{i+1})
\\+&\left( xu_k + \sum_{m=1}^{k-1} d^m \sum_{\substack{2^{k-1} < j < 2^{k} \\ w(\tilde{u}_j)=m+1}} x u_1^{\epsilon_1(j)} \cdots u_{k-1}^{\epsilon_{k-1}(j)} u_k \left( (-xq)^{m-1} + (-xq)^m\right)\right) f_{0_{u_1}}(xq)
\\=& \sum_{i=1}^{k-1} \left( \sum_{m=0}^{k-i-1} d^m \sum_{\substack{1 \leq j < 2^{k-1} \\ w(\tilde{u}_j)=i+m}} x u_1^{\epsilon_1(j)} \cdots u_{k-1}^{\epsilon_{k-1}(j)} \left( (-x)^{m-1} {i+m-1 \brack m-1}_{q} \right. \right.\\
&\qquad  \left. \vphantom{ \sum_{\substack{1 \leq j < 2^{k-1} \\ w(\tilde{u}_j)=i+m}}} \left.  + (-x)^m {i+m \brack m}_{q} \right) \right) \prod_{h=1}^{i-1} \left(1-xq^{h}\right)f_{0_{u_1}}(xq^i)
\\+& \sum_{i=1}^{k-1} \left( \sum_{m=1}^{k-i} d^{m} \sum_{\substack{2^{k-1} < j < 2^{k} \\ w(\tilde{u}_j)=i+m}} x u_1^{\epsilon_1(j)} \cdots u_{k-1}^{\epsilon_{k-1}(j)} u_k \left( (-x)^{m-1} {i+m-2 \brack m-2}_{q} \right. \right.\\
&\qquad  \left. \vphantom{ \sum_{\substack{1 \leq j < 2^{k-1} \\ w(\tilde{u}_j)=i+m}}} \left.  + (-x)^{m} {i+m-1 \brack m-1}_{q} \right) \right) \prod_{h=1}^{i-1} \left(1-xq^{h}\right)f_{0_{u_1}}(xq^i)
\\+&\sum_{i=2}^{k} \left( \sum_{m=0}^{k-i} d^m \sum_{\substack{2^{k-1} < j < 2^{k} \\ w(\tilde{u}_j)=i+m}} x u_1^{\epsilon_1(j)} \cdots u_{k-1}^{\epsilon_{k-1}(j)} u_k \left( (-xq)^{m-1} {i+m-2 \brack m-1}_{q} \right. \right.\\
&\qquad  \left. \vphantom{ \sum_{\substack{1 \leq j < 2^{k-1} \\ w(\tilde{u}_j)=i+m}}} \left.  + (-xq)^m {i+m-1 \brack m}_{q} \right) \right) \prod_{h=1}^{i-1} \left(1-xq^{h}\right)f_{0_{u_1}}(xq^{i})
\\+&\left( xu_k + \sum_{m=1}^{k-1} d^m \sum_{\substack{2^{k-1} < j < 2^{k} \\ w(\tilde{u}_j)=m+1}} x u_1^{\epsilon_1(j)} \cdots u_{k-1}^{\epsilon_{k-1}(j)} u_k \left( (-xq)^{m-1} + (-xq)^m\right)\right) f_{0_{u_1}}(xq)
\\=& \sum_{i=1}^{k-1} \left( \sum_{m=0}^{k-i-1} d^m \sum_{\substack{1 \leq j < 2^{k-1} \\ w(\tilde{u}_j)=i+m}} x u_1^{\epsilon_1(j)} \cdots u_{k-1}^{\epsilon_{k-1}(j)} \left( (-x)^{m-1} {i+m-1 \brack m-1}_{q} \right. \right.\\
&\qquad  \left. \vphantom{ \sum_{\substack{1 \leq j < 2^{k-1} \\ w(\tilde{u}_j)=i+m}}} \left.  + (-x)^m {i+m \brack m}_{q} \right) \right) \prod_{h=1}^{i-1} \left(1-xq^{h}\right)f_{0_{u_1}}(xq^i)
\\+& \sum_{i=1}^{k-1} \left( \sum_{m=1}^{k-i} d^{m} \sum_{\substack{2^{k-1} < j < 2^{k} \\ w(\tilde{u}_j)=i+m}} x u_1^{\epsilon_1(j)} \cdots u_{k-1}^{\epsilon_{k-1}(j)} u_k \left( (-x)^{m-1} {i+m-2 \brack m-2}_{q} \right. \right.\\
&\qquad  \left. \vphantom{ \sum_{\substack{1 \leq j < 2^{k-1} \\ w(\tilde{u}_j)=i+m}}} \left.  + (-x)^{m} {i+m-1 \brack m-1}_{q} \right) \right) \prod_{h=1}^{i-1} \left(1-xq^{h}\right)f_{0_{u_1}}(xq^i)
\\+&\sum_{i=1}^{k} \left( \sum_{m=0}^{k-i} d^m \sum_{\substack{2^{k-1} \leq j < 2^{k} \\ w(\tilde{u}_j)=i+m}} x u_1^{\epsilon_1(j)} \cdots u_{k-1}^{\epsilon_{k-1}(j)} u_k \left( (-xq)^{m-1} {i+m-2 \brack m-1}_{q} \right. \right.\\
&\qquad  \left. \vphantom{ \sum_{\substack{1 \leq j < 2^{k-1} \\ w(\tilde{u}_j)=i+m}}} \left.  + (-xq)^m {i+m-1 \brack m}_{q} \right) \right) \prod_{h=1}^{i-1} \left(1-xq^{h}\right)f_{0_{u_1}}(xq^{i})\\
=& \sum_{i=1}^{k-1} \left[ \sum_{\substack{1 \leq j <2^k \\ w(\tilde{u}_j)=i}}\right.  x u_1^{\epsilon_1(j)} \cdots u_{k-1}^{\epsilon_{k-1}(j)} u_k^{\epsilon_k(j)}
\\&\qquad + \sum_{m=1}^{k-i} d^m  \left(  \sum_{\substack{1 \leq j <2^{k-1} \\ w(\tilde{u}_j)=i+m}} x u_1^{\epsilon_1(j)} \cdots u_{k-1}^{\epsilon_{k-1}(j)} \right. \\
&\qquad \qquad \qquad \qquad \quad \times \left. \left( (-x)^{m-1} {i+m-1 \brack m-1}_{q} + (-x)^m {i+m \brack m}_{q} \right) \right.
\\&\qquad \qquad \qquad \quad +\sum_{\substack{2^{k-1} < j <2^k \\ w(\tilde{u}_j)=i+m}} x u_1^{\epsilon_1(j)} \cdots u_{k-1}^{\epsilon_{k-1}(j)} u_k
\\ &\qquad \qquad\qquad \qquad \qquad \times \Bigg( (-x)^{m-1} \left({i+m-2 \brack m-2}_{q}+ q^{(m-1)}{i+m-2 \brack m-1}_{q} \right) 
\\& \qquad \qquad \qquad \qquad \qquad \quad +\left.\left. (-x)^{m} \left({i+m-1 \brack m-1}_{q} + q^{m} {i+m-1 \brack m}_{q} \right) \Bigg) \vphantom{  \sum_{\substack{1 \leq j <2^k \\ w(\tilde{u}_j)=i}}} \right) \right]
\\&\times  \prod_{h=1}^{i-1} \left(1-xq^{h}\right)f_{0_{u_1}}(xq^{i})
\\&+ xu_1 \cdots u_k \prod_{h=1}^{k-1} \left(1-xq^{h}\right) f_{0_{u_1}}(xq^{k}).
\end{align*}
Then by the first $q$-analogue of Pascal's triangle~\eqref{pascal1},
\begin{align*}
\prod_{i=1}^{k} &\left(1-dxu_i\right) f_{0_{u_1}}(x) - f_{0_{u_{k+1}}}(x)
\\=&\sum_{i=1}^{k} \left( \sum_{m=0}^{k-i} d^m   \sum_{\substack{1 \leq j <2^{k} \\ w(\tilde{u}_j)=i+m}} x u_1^{\epsilon_1(j)} \cdots u_{k-1}^{\epsilon_{k-1}(j)}u_{k}^{\epsilon_{k}(j)} \right. \\
&\qquad \qquad \qquad \qquad \quad \times \left. \left( (-x)^{m-1} {i+m-1 \brack m-1}_{q} + (-x)^m {i+m \brack m}_{q} \right) \vphantom{ \sum_{\substack{1 \leq j <2^{k} \\ w(\tilde{u}_j)=i+m}}} \right)
\\ &\qquad \quad \times  \prod_{h=1}^{i-1} \left(1-xq^{h}\right)f_{0_{u_1}}(xq^{i})
\\=&s_{k+1}(x).
\end{align*}
This completes the proof.
\end{proof}

Starting from Equation~\eqref{eq} for $k=r$, using~\eqref{eq3.7bis} and doing exactly the same computations as above, we obtain the following :
\begin{equation}
\label{eqbis}
\begin{aligned}
\prod_{i=1}^{r} &\left(1-dxu_i\right) f_{0_{u_1}}(x) = f_{1_{u_1}}(x) 
\\ + \sum_{i=1}^{r} &\left( \sum_{m=0}^{r-i} d^m \sum_{\substack{1 \leq j < 2^{r} \\ w(\tilde{u}_j)=i+m}} x u_1^{\epsilon_1(j)} \cdots u_{r}^{\epsilon_{r}(j)} \left( (-x)^{m-1} {i+m-1 \brack m-1}_{q} \right. \right. \\
& \left. \vphantom{\sum_{\substack{1 \leq j < 2^{r} \\ w(\tilde{u}_j)=i+m}}} \left. + (-x)^m {i+m \brack m}_{q} \right) \right) \times \prod_{h=1}^{i-1} \left(1-xq^{h}\right) f_{0_{u_1}}\left(xq^{i}\right).
\end{aligned}
\end{equation}
Finally, using~\eqref{eqf2}, we obtain the desired $q$-difference equation.
\begin{equation}
\label{qdiff}
\tag{$\mathrm{eq}_{r}$}
\begin{aligned}
\prod_{i=1}^{r} &\left(1-dxu_i\right) f_{0_{u_1}}(x) = f_{0_{u_1}}(xq) 
\\ + \sum_{i=1}^{r} &\left( \sum_{m=0}^{r-i} d^m \sum_{\substack{1 \leq j < 2^{r} \\ w(\tilde{u}_j)=i+m}} x u_1^{\epsilon_1(j)} \cdots u_{r}^{\epsilon_{r}(j)} \left( (-x)^{m-1} {i+m-1 \brack m-1}_{q} \right. \right. \\
& \left. \vphantom{\sum_{\substack{1 \leq j < 2^{r} \\ w(\tilde{u}_j)=i+m}}} \left. + (-x)^m {i+m \brack m}_{q} \right) \right) \times \prod_{h=1}^{i-1} \left(1-xq^{h}\right) f_{0_{u_1}}\left(xq^{i}\right).
\end{aligned}
\end{equation}

\subsection{The induction}
Recall we want to find an expression for $f_{0_{u_1}}(1)$, which is the generating function for overpartitions counted by  $E(\ell_1, \dots, \ell_r;k,n)$. We do so by proving the following theorem by induction on $r$.

\begin{theorem}
\label{main}
Let $r$ be a positive integer. Then for every function $f$ satisfying the $q$-difference equation $(\mathrm{eq}_{r})$ and the initial condition $f(0)=1$, we have
$$f(1)= \prod_{k=1}^r \frac{(-u_k;q)_{\infty}}{(du_k;q)_{\infty}}.$$
\end{theorem}

As in~\cite{Doussegene}, we start from a function satisfying $(\mathrm{eq}_{r})$ and do some transformations to obtain a function satisfying $(\mathrm{eq}_{r-1})$ and be able to use the induction hypothesis.

\begin{lemma}
\label{lemmaF}
Let $f$ and $F$ be two functions such that
$$F(x):= f(x) \prod_{n=0}^{\infty} \frac{1-dxu_rq^n}{1-xq^n}.$$
Then $f(0)=1$ and $f$ satisfies~\eqref{qdiff} if and only if $F(0)=1$ and $F$ satisfies the following $q$-difference equation
\begin{equation}
\label{qdiffF}
\tag{$\mathrm{eq}'_{r}$}
\begin{aligned}
&\left(1 + \sum_{i=1}^{r} (-x)^i \left( d^{i-1} \sum_{\substack{ 1 \leq j < 2^{r-1} \\  w(\tilde{u}_j)=i-1}} u_1^{\epsilon_1(j)} \cdots u_{r-1}^{\epsilon_{r-1}(j)}  +d^i \sum_{\substack{  1 \leq j < 2^{r-1} \\  w(\tilde{u}_j)=i}} u_1^{\epsilon_1(j)} \cdots u_{r-1}^{\epsilon_{r-1}(j)} \right) \right) F(x)
\\&= F\left(xq\right) + \sum_{i=1}^r \sum_{\ell=1}^r \sum_{k=0}^{\min(i-1,\ell-1)}c_{k,i}b_{\ell-k,j} (-1)^{\ell-1}x^{\ell} F\left(xq^{i}\right),
\end{aligned}
\end{equation}
where
$$c_{k,i}:= d^k u_r^k q^{\frac{k(k+1)}{2}} {i-1 \brack k}_{q} ,$$
and
$$b_{m,i}:= \left( d^{m-1} \sum_{\substack{ 1 \leq j < 2^{r}\\  w(\tilde{u}_j)=i+m-1}} u_1^{\epsilon_1(j)} \cdots u_{r}^{\epsilon_{r}(j)} + d^m \sum_{\substack{1 \leq j < 2^{r} \\  w(\tilde{u}_j)=i+m}} u_1^{\epsilon_1(j)} \cdots u_{r}^{\epsilon_{r}(j)} \right) {i+m-1 \brack m-1}_{q}.$$
\end{lemma}
\begin{proof}
Starting from $(\mathrm{eq}_{r})$ and writing $f$ in terms of $F$, we obtain
\begin{align*}
(1-x) &\prod_{i=1}^{r-1} \left(1-dxu_i\right) F(x) = F(xq) 
\\ + \sum_{i=1}^{r} &\left( \sum_{m=0}^{r-i} d^m \sum_{\substack{1 \leq j < 2^{r} \\ w(\tilde{u}_j)=i+m}} x u_1^{\epsilon_1(j)} \cdots u_{r}^{\epsilon_{r}(j)} \left( (-x)^{m-1} {i+m-1 \brack m-1}_{q} \right. \right. \\
& \left. \vphantom{\sum_{\substack{1 \leq j < 2^{r} \\ w(\tilde{u}_j)=i+m}}} \left. + (-x)^m {i+m \brack m}_{q} \right) \right) \times \prod_{h=1}^{i-1} \left(1-dxu_rq^{h}\right)F\left(xq^{i}\right).
\end{align*}
Using the conventions
$$\sum_{\substack{1 \leq j < 2^{r} \\ w(\tilde{u}_j)=n}} u_1^{\epsilon_1(j)} \cdots u_{r}^{\epsilon_{r}(j)} = 0 \ \text{for}\ n > r,$$
and
$$\sum_{\substack{1 \leq j < 2^{r} \\ w(\tilde{u}_j)=0}} x u_1^{\epsilon_1(j)} \cdots u_{r}^{\epsilon_{r}(j)} = 1,$$
together with the $q$-binomial theorem (see for example~\cite{Gasper}), this can ban rewritten as
\begin{align*}
&\left(1 + \sum_{i=1}^{r} (-x)^i \left( d^{i-1} \sum_{\substack{ 1 \leq j < 2^{r-1} \\  w(\tilde{u}_j)=i-1}} u_1^{\epsilon_1(j)} \cdots u_{r-1}^{\epsilon_{r-1}(j)} \right.\right. \\
&\qquad\qquad \qquad \quad \left.\left.+d^i \sum_{\substack{  1 \leq j < 2^{r-1} \\  w(\tilde{u}_j)=i}} u_1^{\epsilon_1(j)} \cdots u_{r-1}^{\epsilon_{r-1}(j)} \right) \right) F(x)= F(xq) 
\\ &+ \sum_{i=1}^{r} \left( \sum_{m=1}^{r-i+1} \left(d^{m-1} \sum_{\substack{ 1 \leq j < 2^{r} \\  w(\tilde{u}_j)=i+m-1}} u_1^{\epsilon_1(j)} \cdots u_{r}^{\epsilon_{r}(j)} \right.\right.\\
&\qquad \qquad \qquad \quad \left. \left. + d^m \sum_{\substack{ 1 \leq j < 2^{r} \\  w(\tilde{u}_j)=i+m}} u_1^{\epsilon_1(j)} \cdots u_{r}^{\epsilon_{r}(j)}\right) {i+m-1 \brack m-1}_{q} (-1)^{m-1} x^m \right)
\\ &\times \left(\sum_{k=0}^{i-1}  d^k (-x)^ku_r^k q^{\frac{k(k-1)}{2}} {i-1 \brack k}_{q}\right) F\left(xq^{i}\right).
\end{align*}

Expanding and noting that $b_{\ell-k,i} = 0$ if $i+\ell-k-1 \geq r$, we obtain~\eqref{qdiffF}.
Moreover, $F(0)=f(0)=1$ and the lemma is proved.
\end{proof}

We can now transform ~\eqref{qdiffF} into a recurrence equation on the coefficients of $F$ as a power series in $x$.
\begin{lemma}
\label{lemmaA}
Let $F$ be a function and $(A_n)_{n \in \N}$ a sequence such that $$F(x) =: \sum_{n=0}^{\infty} A_n x^n.$$
Then $F$ satisfies $(\mathrm{eq}'_{r})$ and the initial condition $F(0)=1$ if and only if $A_0=1$ and $(A_n)_{n \in \N}$ satisfies the following recurrence equation
\begin{equation}
\label{recA}
\tag{$\mathrm{rec}_{r}$}
\begin{aligned}
& \left(1-q^{n}\right) A_n = \sum_{m=1}^{r} \left( d^{m-1} \sum_{\substack{ 1 \leq j < 2^{r-1} \\  w(\tilde{u}_j)=m-1}} u_1^{\epsilon_1(j)} \cdots u_{r-1}^{\epsilon_{r-1}(j)} \right.\\ 
&\left. \qquad  \qquad \qquad \qquad + d^{m}  \sum_{\substack{ 1 \leq j < 2^{r-1} \\  w(\tilde{u}_j)=m}} u_1^{\epsilon_1(j)} \cdots u_{r-1}^{\epsilon_{r-1}(j)} \right.\\
&\qquad \qquad \qquad \qquad+ \left. \sum_{i=1}^r \sum_{k=0}^{\min(i-1,m-1)} c_{k,i} b_{m-k,i}q^{i(n-m)} \right) (-1)^{m+1} A_{n-m}.
\end{aligned}
\end{equation}
\end{lemma}
\begin{proof}
By the definition of $(A_n)_{n \in \N}$ and~\eqref{qdiffF},
\begin{align*}
\left(1-q^{n}\right) A_n =& \sum_{i=1}^{r} \left( d^{i-1} \sum_{\substack{ 1 \leq j < 2^{r-1} \\  w(\tilde{u}_j)=i-1}} u_1^{\epsilon_1(j)} \cdots u_{r-1}^{\epsilon_{r-1}(j)} \right.\\
&\qquad \left.+ d^{j}\sum_{\substack{ 1 \leq j < 2^{r-1} \\  w(\tilde{u}_j)=i}} u_1^{\epsilon_1(j)} \cdots u_{r-1}^{\epsilon_{r-1}(j)} \right) (-1)^{i+1} A_{n-i}
\\ &+  \sum_{i=1}^r \sum_{\ell=1}^r \sum_{k=0}^{\min(i-1,\ell-1)} c_{k,i} b_{\ell-k,i}q^{i(n-\ell)} (-1)^{\ell+1} A_{n-l}.
\end{align*}
Factorising leads to~\eqref{recA}, and $A_n = F(0)=1$.
\end{proof}

For convenience, we now do transformations starting from $(\mathrm{eq}_{r-1})$.

\begin{lemma}
\label{lemmaG}
Let $g$ and $G$ be two functions such that
$$G(x):= g(x) \prod_{n=0}^{\infty} \frac{1}{1-xq^{n}}.$$
Then $g$ satisfies $(\mathrm{eq}_{r-1})$ and $g(0)=1$ if and only if $G(0)=1$ and $G$ satisfies the following $q$-difference equation
\begin{equation}
\label{qdiffG}
\tag{$\mathrm{eq}''_{r-1}$}
\begin{aligned}
&\left(1 + \sum_{i=1}^{r} \left( d^{i-1} \sum_{\substack{  1 \leq j < 2^{r-1} \\  w(\tilde{u}_j)=i-1}} u_1^{\epsilon_1(j)} \cdots u_{r-1}^{\epsilon_{r-1}(j)} \right.\right. \\
&\qquad \qquad \quad \left.\left. +d^i \sum_{\substack{ 1 \leq j < 2^{r-1} \\  w(\tilde{u}_j)=i}} u_1^{\epsilon_1(j)} \cdots u_{r-1}^{\epsilon_{r-1}(j)} \right) (-x)^i \right) G(x)\\
&= G\left(xq\right) + \sum_{i=1}^r \sum_{m=1}^{r-i} \left(d^{m-1} \sum_{\substack{  1 \leq j < 2^{r-1} \\  w(\tilde{u}_j)=i+m-1}} u_1^{\epsilon_1(j)} \cdots u_{r-1}^{\epsilon_{r-1}(j)} \right.
\\& \qquad \qquad \left.+d^m \sum_{\substack{ 1 \leq j < 2^{r-1} \\  w(\tilde{u}_j)=i+m}} u_1^{\epsilon_1(j)} \cdots u_{r-1}^{\epsilon_{r-1}(j)} \right) {i+m-1 \brack m-1}_{q} (-1)^{m+1} x^m  G\left(xq^{i}\right).
\end{aligned}
\end{equation}
\end{lemma}
\begin{proof}
By the definition of $G$ and $(\mathrm{eq}_{r-1})$, we have
\begin{align*}
&(1-x) \prod_{i=1}^{r-1} \left(1-dxu_i\right) G(x) = G(xq) 
\\&+ \sum_{i=1}^{r-1} \left( \sum_{m=0}^{r-i-1} d^m \sum_{\substack{  1 \leq j < 2^{r-1} \\  w(\tilde{u}_j)=i+m}} x u_1^{\epsilon_1(j)} \cdots u_{r-1}^{\epsilon_{r-1}(j)} \right.\\
& \qquad \qquad \times \left. \vphantom{\sum_{\substack{  1 \leq j < 2^{r-1} \\  w(\tilde{u}_j)=i+m-1}}} \left( (-x)^{m-1} {i+m-1 \brack m-1}_{q} + (-x)^m {i+m \brack m}_{q} \right) \right) G\left(xq^{i}\right).
\end{align*}
Then, using the $q$-binomial theorem as in the proof of Lemma~\ref{lemmaF}, this can be reformulated as~\eqref{qdiffG}, and $G(0)= g(0)=1.$
\end{proof}

Again we want to translate this into a recurrence equation on the coefficients of $G$ written as a power series in the variable $x$.

\begin{lemma}
\label{lemmaa}
Let $G$ be a function and $(a_n)_{n \in \N}$ be a sequence such that
$$G(x) =: \sum_{n=0}^{\infty} a_n x^n.$$
Then $G$ satisfies $(\mathrm{eq}''_{r-1})$ and $G(0)=1$ if and only if $a_0=1$ and $(a_n)_{n \in \N}$ satisfies the following recurrence equation
\begin{equation}
\label{reca}
\tag{$\mathrm{rec''}_{r-1}$}
\begin{aligned}
& \left(1-q^{n}\right) a_n = \sum_{m=1}^{r} \sum_{i=0}^{r-1} \left(d^{m-1} \sum_{\substack{  1 \leq j < 2^{r-1} \\  w(\tilde{u}_j)=i+m-1}}u_1^{\epsilon_1(j)} \cdots u_{r-1}^{\epsilon_{r-1}(j)} \right.\\
&\left. + d^{m} \sum_{\substack{  1 \leq j < 2^{r-1} \\  w(\tilde{u}_j)=i+m}} u_1^{\epsilon_1(j)} \cdots u_{r-1}^{\epsilon_{r-1}(j)} \right) {i+m-1 \brack m-1}_{q} q^{i(n-m)} (-1)^{m+1} a_{n-m}.
\end{aligned}
\end{equation}
\end{lemma}
\begin{proof}
Plugging the definition of $(a_n)_{n \in \N}$ into~\eqref{qdiffG} gives
\begin{align*}
& \left(1-q^{n}\right) a_n = \sum_{m=1}^{r} \left( d^{m-1} \sum_{\substack{  1 \leq j < 2^{r-1} \\  w(\tilde{u}_j)=m-1}} u_1^{\epsilon_1(j)} \cdots u_{r-1}^{\epsilon_{r-1}(j)} \right.\\
&\left. \qquad \qquad \qquad \qquad \quad + d^{m} \sum_{\substack{  1 \leq j < 2^{r-1} \\  w(\tilde{u}_j)=m}} u_1^{\epsilon_1(j)} \cdots u_{r-1}^{\epsilon_{r-1}(j)} \right) (-1)^{m+1} a_{n-m}
\\ &+ \sum_{m=1}^{r-1} \sum_{i=1}^{r-1} \left(d^{m-1} \sum_{\substack{  1 \leq j < 2^{r-1} \\  w(\tilde{u}_j)=i+m-1}}u_1^{\epsilon_1(j)} \cdots u_{r-1}^{\epsilon_{r-1}(j)} \right.\\
&\left. \qquad \qquad \quad + d^{m} \sum_{\substack{  1 \leq j < 2^{r-1} \\  w(\tilde{u}_j)=i+m}} u_1^{\epsilon_1(j)} \cdots u_{r-1}^{\epsilon_{r-1}(j)} \right) {i+m-1 \brack m-1}_{q} q^{i(n-m)} (-1)^{m+1} a_{n-m}.
\end{align*}
Gathering the sums and noting that $a_n = G(0)=1$ completes the proof.
\end{proof}

We now do a final transformation and obtain a last recurrence equation.

\begin{lemma}
\label{lemmaA'}
Let $(a_n)_{n \in \N}$ and $(A'_n)_{n \in \N}$ be two sequences such that
$$ A'_n := a_n \prod_{k=0}^{n-1} \left( 1 +u_rq^{k} \right).$$
Then $(a_n)_{n \in \N}$ satisfies $(\mathrm{rec}''_{r-1})$ and the initial condition $a_0=1$ if and only if $A'_0=1$ and $(A'_n)_{n \in \N}$ satisfies the following recurrence equation
\begin{equation}
\label{recA'}
\tag{$\mathrm{rec'}_{r-1}$}
\begin{aligned}
\left(1-q^{n}\right) A'_n =& \sum_{m=1}^{r} \left( \sum_{\nu=0}^{r-1} \sum_{\mu=0}^{\min(m-1, \nu)} f_{m,\mu} e_{m,\nu - \mu} q^{\nu (n-m)} \right.
\\&+ \left. u_r \sum_{\nu=1}^{r} \sum_{\mu=0}^{\min(m-1, \nu-1)} f_{m,\mu} e_{m,\nu - \mu -1} q^{\nu (n-m)}\right) (-1)^{m+1}  A'_{n-m},
\end{aligned}
\end{equation}
where
\begin{align*}
e_{m,i} := &\left(d^{m-1} \sum_{\substack{  1 \leq j < 2^{r-1} \\  w(\tilde{u}_j)=i+m-1}}u_1^{\epsilon_1(j)} \cdots u_{r-1}^{\epsilon_{r-1}(j)} \right. \\
&\left. + d^{m} \sum_{\substack{  1 \leq j < 2^{r-1} \\  w(\tilde{u}_j)=i+m}} u_1^{\epsilon_1(j)} \cdots u_{r-1}^{\epsilon_{r-1}(j)} \right) {i+m-1 \brack m-1}_{q},
\end{align*}
and
$$f_{m,k} := u_r^k q^{\frac{k(k+1)}{2}} {m-1 \brack k}_{q}.$$
\end{lemma}
\begin{proof}
Replacing the definition of $(A'_n)_{n \in \N}$ into~\eqref{reca}, we have
\begin{align*}
\left(1-q^{n}\right) A'_n = &\sum_{m=1}^{r} \sum_{i=0}^{r-1} \left(d^{m-1} \sum_{\substack{  1 \leq j < 2^{r-1} \\  w(\tilde{u}_j)=i+m-1}}u_1^{\epsilon_1(j)} \cdots u_{r-1}^{\epsilon_{r-1}(j)} \right.\\
&\left. \qquad \qquad + d^{m} \sum_{\substack{  1 \leq j < 2^{r-1} \\  w(\tilde{u}_j)=i+m}} u_1^{\epsilon_1(j)} \cdots u_{r-1}^{\epsilon_{r-1}(j)} \right)\\
&\times {i+m-1 \brack m-1}_{q} q^{i(n-m)} (-1)^{m+1} \prod_{k=1}^m \left(1 + u_rq^{n-k}\right) A'_{n-m}.
\end{align*}
Furthermore, by a change of variables,
\begin{align*}
\prod_{k=1}^m \left(1 + u_rq^{n-k}\right) &= \prod_{k=0}^{m-1} \left(1 + u_rq^{k+n-m}\right)
\\&= \left( 1+u_rq^{n-m}\right) \prod_{k=1}^{m-1} \left(1 + u_rq^{k+n-m}\right)
\\&= \left( 1+u_rq^{n-m}\right) \sum_{k=0}^{m-1} u_r^kq^{\frac{k(k+1)}{2} +k(n-m)} {m-1 \brack k}_{q},
\end{align*}
where we used the $q$-binomial theorem to obtain the last equality.
Thus
\begin{align*}
&\left(1-q^{n}\right) A'_n = 
\\& \sum_{m=1}^{r} \sum_{i=0}^{r-1} \left(d^{m-1} \sum_{\substack{  1 \leq j < 2^{r-1} \\  w(\tilde{u}_j)=i+m-1}}u_1^{\epsilon_1(j)} \cdots u_{r-1}^{\epsilon_{r-1}(j)} + d^{m} \sum_{\substack{  1 \leq j < 2^{r-1} \\  w(\tilde{u}_j)=i+m}} u_1^{\epsilon_1(j)} \cdots u_{r-1}^{\epsilon_{r-1}(j)} \right)\\
&\times {i+m-1 \brack m-1}_{q} q^{i(n-m)} \left( 1+u_rq^{n-m}\right) \sum_{k=0}^{m-1} u_r^kq^{\frac{k(k+1)}{2} +k(n-m)} {m-1 \brack k}_{q} (-1)^{m+1} A'_{n-m}\\
&= \sum_{m=1}^{r} \left[ \sum_{i=0}^{r-1} \left(d^{m-1} \sum_{\substack{  1 \leq j < 2^{r-1} \\  w(\tilde{u}_j)=i+m-1}}u_1^{\epsilon_1(j)} \cdots u_{r-1}^{\epsilon_{r-1}(j)} + d^{m} \sum_{\substack{  1 \leq j < 2^{r-1} \\  w(\tilde{u}_j)=i+m}} u_1^{\epsilon_1(j)} \cdots u_{r-1}^{\epsilon_{r-1}(j)} \right) \right.\\
&\left.\qquad \qquad \times {i+m-1 \brack m-1}_{q} q^{i(n-m)} \sum_{k=0}^{m-1} u_r^kq^{\frac{k(k+1)}{2} +k(n-m)} {m-1 \brack k}_{q} \right.
\\& \qquad \left. +\sum_{i=0}^{r-1} \left(d^{m-1} \sum_{\substack{  1 \leq j < 2^{r-1} \\  w(\tilde{u}_j)=i+m-1}}u_1^{\epsilon_1(j)} \cdots u_{r-1}^{\epsilon_{r-1}(j)} + d^{m} \sum_{\substack{  1 \leq j < 2^{r-1} \\  w(\tilde{u}_j)=i+m}} u_1^{\epsilon_1(j)} \cdots u_{r-1}^{\epsilon_{r-1}(j)} \right) \right.\\
&\left. \qquad \times u_r {i+m-1 \brack m-1}_{q} q^{(i+1)(n-m)} \sum_{k=0}^{m-1} u_r^kq^{\frac{k(k+1)}{2} +k(n-m)} {m-1 \brack k}_{q} \vphantom{\sum_{\substack{  1 \leq j < 2^{r-1} \\  w(\tilde{u}_j)=i+m-1}}} \right] (-1)^{m+1} A'_{n-m}.
\end{align*}
Therefore
\begin{align*}
\left(1-q^{n}\right) A'_n =\sum_{m=1}^{r} &\left(\sum_{i=0}^{r-1} e_{m,i} q^{i(n-m)} \sum_{k=0}^{m-1} f_{m,k} q^{k(n-m)} \right.
\\&+ \left. u_r \sum_{i=1}^{r} e_{m,i-1} q^{i(n-m)} \sum_{k=0}^{m-1} f_{m,k} q^{k(n-m)} \right) (-1)^{m+1} A'_{n-m}.
\end{align*}
Expanding gives~\eqref{recA'}, and $A'_0 = a_0=1$.
\end{proof}

The key step is now to show that $(A_n)_{n \in \N}$ and $(A'_n)_{n \in \N}$ are equal.

\begin{lemma}
\label{equalAA'}
Let $(A_n)_{n \in \N}$ and $(A'_n)_{n \in \N}$ be defined as in Lemmas~\ref{lemmaA} and~\ref{lemmaA'}.
Then for every $n \in \N$, $A_n = A'_n$.
\end{lemma}
\begin{proof}
To prove the equality, we show that for every $1 \leq m \leq r,$ the coefficient of $(-1)^{m+1} A_{n-m}$ in~\eqref{recA} and the coefficient of $(-1)^{m+1} A'_{n-m}$ in~\eqref{recA'} are equal.
Let $m \in \lbrace 1,...,r \rbrace$ and
\begin{align*}
S_{m} &: = \left[ (-1)^{m+1} A_{n-m}\right] (\mathrm{rec}_{r}) 
\\&= d^{m-1} \sum_{\substack{  1 \leq j < 2^{r-1} \\  w(\tilde{u}_j)=m-1}} u_1^{\epsilon_1(j)} \cdots u_{r-1}^{\epsilon_{r-1}(j)} + d^{m} \sum_{\substack{  1 \leq j < 2^{r-1} \\  w(\tilde{u}_j)=m}} u_1^{\epsilon_1(j)} \cdots u_{r-1}^{\epsilon_{r-1}(j)} \\
&+ \sum_{i=1}^r \sum_{k=0}^{\min(i-1,m-1)} c_{k,i} b_{m-k,i} q^{i(n-m)}
\end{align*}
and
\begin{align*}
S'_{m} &: = \left[ (-1)^{m+1} A'_{n-m}\right] (\mathrm{rec'}_{r-1}) 
\\&= \sum_{\nu=0}^{r-1} \sum_{\mu=0}^{\min(m-1, \nu)} f_{m,\mu} e_{m,\nu - \mu} q^{\nu (n-m)} + u_r \sum_{\nu=1}^{r} \sum_{\mu=0}^{\min(m-1, \nu-1)} f_{m,\mu} e_{m,\nu - \mu -1} q^{\nu(n-m)}
\\&= f_{m,0} e_{m,0} + \sum_{\nu=1}^{r} \left( \sum_{\mu=0}^{\min(m-1, \nu)} f_{m,\mu} e_{m,\nu - \mu} + u_r \sum_{\mu=0}^{\min(m-1, \nu-1)} f_{m,\mu} e_{m,\nu - \mu -1} \right) q^{\nu (n-m)},
\end{align*}
because $e_{m, r-\mu}=0$ for all $\mu$.

We start by noting that $$f_{m,0} e_{m,0} = d^{m-1} \sum_{\substack{  1 \leq j < 2^{r-1} \\  w(\tilde{u}_j)=m-1}} u_1^{\epsilon_1(j)} \cdots u_{r-1}^{\epsilon_{r-1}(j)} + d^{m} \sum_{\substack{  1 \leq j < 2^{r-1} \\  w(\tilde{u}_j)=m}} u_1^{\epsilon_1(j)} \cdots u_{r-1}^{\epsilon_{r-1}(j)}.$$
Now define
$$T_{m,i}:= \sum_{k=0}^{\min(i-1,m-1)} c_{k,i} b_{m-k,i},$$
and
$$T'_{m,i}:= \sum_{k=0}^{\min(m-1, i)} f_{m,k} e_{m,i-k} + u_r \sum_{k=0}^{\min(m-1,i-1)} f_{m,k} e_{m,i-k-1}.$$
So it only remains to show that for all $1 \leq i \leq r$, $$T_{m,i}= T'_{m,i}.$$
We have
\begin{equation}
\label{eqcb}
\begin{aligned}
&c_{k,i} b_{m-k,i} 
\\&= u_r^k q^{\frac{k(k+1)}{2}} {i-1 \brack k}_{q}{i+m-k-1 \brack m-k-1}_{q} \\
&\times \left( d^{m-1}  \sum_{\substack{  1 \leq j < 2^{r} \\  w(\tilde{u}_j)=i+m-k-1}} u_1^{\epsilon_1(j)} \cdots u_{r}^{\epsilon_{r}(j)} +d^m \sum_{\substack{  1 \leq j < 2^{r} \\  w(\tilde{u}_j)=i+m-k}} u_1^{\epsilon_1(j)} \cdots u_{r}^{\epsilon_{r}(j)}\right) 
\\&= u_r^k q^{\frac{k(k+1)}{2}} {i-1 \brack k}_{q}  {i+m-k-1 \brack m-k-1}_{q}\\
& \times \left( d^{m-1} \sum_{\substack{  1 \leq j < 2^{r-1} \\  w(\tilde{u}_j)=i+m-k-1}} u_1^{\epsilon_1(j)} \cdots u_{r-1}^{\epsilon_{r-1}(j)} +d^m \sum_{\substack{  1 \leq j < 2^{r-1} \\  w(\tilde{u}_j)=i+m-k}} u_1^{\epsilon_1(j)} \cdots u_{r-1}^{\epsilon_{r-1}(j)} \right)
\\&+u_r^{k+1}q^{\frac{k(k+1)}{2}} {i-1 \brack k}_{q}  {i+m-k-1 \brack m-k-1}_{q} 
\\&\times \left( d^{m-1} \sum_{\substack{  1 \leq j < 2^{r-1} \\  w(\tilde{u}_j)=i+m-k-2}} u_1^{\epsilon_1(j)} \cdots u_{r-1}^{\epsilon_{r-1}(j)} +d^m \sum_{\substack{  1 \leq j < 2^{r-1} \\  w(\tilde{u}_j)=i+m-k-1}} u_1^{\epsilon_1(j)} \cdots u_{r-1}^{\epsilon_{r-1}(j)} \right)  ,
\end{aligned}
\end{equation}
where the last equality follows from splitting the sum according to whether $\tilde{u}_j$ contains $u_r$ as a primary colour or not.

On the other hand one has
\begin{equation}
\label{eqfe}
\begin{aligned}
&f_{m,k} e_{m,i-k} = u_r^kq^{\frac{k(k+1)}{2}} {m-1 \brack k}_{q} {i+m-k-1 \brack m-1}_{q} \\
&\times \left( d^{m-1} \sum_{\substack{  1 \leq j < 2^{r-1} \\  w(\tilde{u}_j)=i+m-k-1}} u_1^{\epsilon_1(j)} \cdots u_{r-1}^{\epsilon_{r-1}(j)}+d^m \sum_{\substack{  1 \leq j < 2^{r-1} \\  w(\tilde{u}_j)=i+m-k}} u_1^{\epsilon_1(j)} \cdots u_{r-1}^{\epsilon_{r-1}(j)} \right) ,
\end{aligned}
\end{equation}
and
\begin{equation}
\label{eqfe'}
\begin{aligned}
&u_r f_{m,k} e_{m,i-k-1} = u_r^{k+1}q^{\frac{k(k+1)}{2}} {m-1 \brack k}_{q} {i+m-k-2 \brack m-1}_{q}
\\&\times \left( d^{m-1} \sum_{\substack{  1 \leq j < 2^{r-1} \\  w(\tilde{u}_j)=i+m-k-2}} u_1^{\epsilon_1(j)} \cdots u_{r-1}^{\epsilon_{r-1}(j)}+d^m \sum_{\substack{  1 \leq j < 2^{r-1} \\  w(\tilde{u}_j)=i+m-k-1}} u_1^{\epsilon_1(j)} \cdots u_{r-1}^{\epsilon_{r-1}(j)} \right) .
\end{aligned}
\end{equation}

For all $j,k,m \in \N$, we have the following equality:
\begin{equation}
\label{equalityqbin}
{m-1 \brack k}_{q^N} {i+m-k-1 \brack m-1}_{q^N} = {i \brack k}_{q^N} {i+m-k-1 \brack m-k-1}_{q^N}.
\end{equation}
Using~\eqref{equalityqbin}, we obtain
\begin{align*}
T'_{m,i}&= \chi( i \leq m-1) \ u_r^iq^{\frac{i(i+1)}{2}}{m-1 \brack m-i-1}_{q} \\
&\times \left( d^{m-1} \sum_{\substack{  1 \leq j < 2^{r-1} \\  w(\tilde{u}_j)=m-1}} u_1^{\epsilon_1(j)} \cdots u_{r-1}^{\epsilon_{r-1}(j)} +d^m \sum_{\substack{  1 \leq j < 2^{r-1} \\  w(\tilde{u}_j)=m}} u_1^{\epsilon_1(j)} \cdots u_{r-1}^{\epsilon_{r-1}(j)}\right) 
\\&+ \sum_{k=0}^{\min(m-1,i-1)} u_r^k q^{\frac{k(k+1)}{2}} {i \brack k}_{q} {i+m-k-1 \brack m-k-1}_{q}
\\&\times \left( d^{m-1} \sum_{\substack{  1 \leq j < 2^{r-1} \\  w(\tilde{u}_j)=i+m-k-1}} u_1^{\epsilon_1(j)} \cdots u_{r-1}^{\epsilon_{r-1}(j)} +d^m \sum_{\substack{  1 \leq j < 2^{r-1} \\  w(\tilde{u}_j)=i+m-k}} u_1^{\epsilon_1(j)} \cdots u_{r-1}^{\epsilon_{r-1}(j)}\right)
\\&+ \sum_{k=0}^{\min(m-1,i-1)} u_r^{k+1} q^{\frac{k(k+1)}{2}} {i-1 \brack k}_{q} {i+m-k-2 \brack m-k-1}_{q}
\\&\times \left( d^{m-1} \sum_{\substack{  1 \leq j < 2^{r-1} \\  w(\tilde{u}_j)=i+m-k-2}} u_1^{\epsilon_1(j)} \cdots u_{r-1}^{\epsilon_{r-1}(j)} +d^m \sum_{\substack{  1 \leq j < 2^{r-1} \\  w(\tilde{u}_j)=i+m-k-1}} u_1^{\epsilon_1(j)} \cdots u_{r-1}^{\epsilon_{r-1}(j)}\right).
\end{align*}
By the second $q$-anlogue~\eqref{pascal2} of Pascal's triangle, we have
$${i \brack k}_{q} = {i-1 \brack k}_{q} + q^{i-k} {i-1 \brack k-1}_{q},$$
$${i+m-k-2 \brack m-k-1}_{q} = {i+m-k-1 \brack m-k-1}_{q} - q^{i} {i+m-k-2 \brack m-k-2}_{q}.$$
Thus we can rewrite $T'_{m,i}$ as
\begin{align*}
T'_{m,i}&= \chi( i \leq m-1) \ u_r^iq^{\frac{i(i+1)}{2}}{m-1 \brack m-i-1}_{q} \\
&\times \left( d^{m-1} \sum_{\substack{  1 \leq j < 2^{r-1} \\  w(\tilde{u}_j)=m-1}} u_1^{\epsilon_1(j)} \cdots u_{r-1}^{\epsilon_{r-1}(j)} +d^m \sum_{\substack{  1 \leq j < 2^{r-1} \\  w(\tilde{u}_j)=m}} u_1^{\epsilon_1(j)} \cdots u_{r-1}^{\epsilon_{r-1}(j)}\right)
\\&+ \sum_{k=0}^{\min(m-1,i-1)} u_r^k q^{\frac{k(k+1)}{2}} {i-1 \brack k}_{q} {i+m-k-1 \brack m-k-1}_{q}
\\&\times \left( d^{m-1} \sum_{\substack{  1 \leq j < 2^{r-1} \\  w(\tilde{u}_j)=i+m-k-1}} u_1^{\epsilon_1(j)} \cdots u_{r-1}^{\epsilon_{r-1}(j)} +d^m \sum_{\substack{  1 \leq j < 2^{r-1} \\  w(\tilde{u}_j)=i+m-k}} u_1^{\epsilon_1(j)} \cdots u_{r-1}^{\epsilon_{r-1}(j)}\right)
\\&+ \sum_{k=0}^{\min(m-1,i-1)} u_r^k q^{\frac{k(k-1)}{2}+j} {i-1 \brack k-1}_{q} {i+m-k-1 \brack m-k-1}_{q}
\\&\times \left( d^{m-1} \sum_{\substack{  1 \leq j < 2^{r-1} \\  w(\tilde{u}_j)=i+m-k-1}} u_1^{\epsilon_1(j)} \cdots u_{r-1}^{\epsilon_{r-1}(j)} +d^m \sum_{\substack{  1 \leq j < 2^{r-1} \\  w(\tilde{u}_j)=i+m-k}} u_1^{\epsilon_1(j)} \cdots u_{r-1}^{\epsilon_{r-1}(j)}\right)
\\&+ \sum_{k=0}^{\min(m-1,i-1)} u_r^{k+1} q^{\frac{k(k+1)}{2}} {i-1 \brack k}_{q} {i+m-k-1 \brack m-k-1}_{q}
\\&\times \left( d^{m-1} \sum_{\substack{  1 \leq j < 2^{r-1} \\  w(\tilde{u}_j)=i+m-k-2}} u_1^{\epsilon_1(j)} \cdots u_{r-1}^{\epsilon_{r-1}(j)} +d^m \sum_{\substack{  1 \leq j < 2^{r-1} \\  w(\tilde{u}_j)=i+m-k-1}} u_1^{\epsilon_1(j)} \cdots u_{r-1}^{\epsilon_{r-1}(j)}\right)
\\&- \sum_{k=0}^{\min(m-2,i-1)} u_r^{k+1} q^{\frac{k(k+1)}{2}+j} {i-1 \brack k}_{q} {i+m-k-2 \brack m-k-2}_{q}
\\&\times \left( d^{m-1} \sum_{\substack{  1 \leq j < 2^{r-1} \\  w(\tilde{u}_j)=i+m-k-2}} u_1^{\epsilon_1(j)} \cdots u_{r-1}^{\epsilon_{r-1}(j)} +d^m \sum_{\substack{  1 \leq j < 2^{r-1} \\  w(\tilde{u}_j)=i+m-k-1}} u_1^{\epsilon_1(j)} \cdots u_{r-1}^{\epsilon_{r-1}(j)}\right).
\end{align*}
By~\eqref{eqcb}, the sum of the second and fourth terms above is equal to $T_{m,j}$.
Let $X$ denote the sum of the first, third and fifth terms. It remains to show that $X=0$.
By the change of variable $k'=k+1$ in the last sum, we obtain
\begin{align*}
X &= \chi( i \leq m-1) \ u_r^iq^{\frac{i(i+1)}{2}}{m-1 \brack m-i-1}_{q} \\
&\times \left( d^{m-1} \sum_{\substack{  1 \leq j < 2^{r-1} \\  w(\tilde{u}_j)=m-1}} u_1^{\epsilon_1(j)} \cdots u_{r-1}^{\epsilon_{r-1}(j)} +d^m \sum_{\substack{  1 \leq j < 2^{r-1} \\  w(\tilde{u}_j)=m}} u_1^{\epsilon_1(j)} \cdots u_{r-1}^{\epsilon_{r-1}(j)}\right)\\
&+\sum_{k=0}^{\min(m-1,i-1)} u_r^k q^{\frac{k(k-1)}{2}+j} {i-1 \brack k-1}_{q} {i+m-k-1 \brack m-k-1}_{q}
\\&\times \left( d^{m-1} \sum_{\substack{  1 \leq j < 2^{r-1} \\  w(\tilde{u}_j)=i+m-k-1}} u_1^{\epsilon_1(j)} \cdots u_{r-1}^{\epsilon_{r-1}(j)} +d^m \sum_{\substack{  1 \leq j < 2^{r-1} \\  w(\tilde{u}_j)=i+m-k}} u_1^{\epsilon_1(j)} \cdots u_{r-1}^{\epsilon_{r-1}(j)}\right)\\
&-\sum_{k=1}^{\min(m-1,i)} u_r^k q^{\frac{k(k-1)}{2}+j} {i-1 \brack k-1}_{q} {i+m-k-1 \brack m-k-1}_{q}
\\&\times \left( d^{m-1} \sum_{\substack{  1 \leq j < 2^{r-1} \\  w(\tilde{u}_j)=i+m-k-1}} u_1^{\epsilon_1(j)} \cdots u_{r-1}^{\epsilon_{r-1}(j)} +d^m \sum_{\substack{  1 \leq j < 2^{r-1} \\  w(\tilde{u}_j)=i+m-k}} u_1^{\epsilon_1(j)} \cdots u_{r-1}^{\epsilon_{r-1}(j)}\right)
\\&= 0,
\end{align*}
because the two sums are equal when $i > m-1$, and when $i \leq m-1$, the remaining term cancels with the first term.
This completes the proof.
\end{proof}

We can finally use all the previous lemmas to prove Theorem~\ref{main}.

\begin{proof}[Proof of Theorem~\ref{main}]
Let us start by the initial case $r=1$. Let $f$ such that $f(0)=1$ and
\begin{equation}
\tag{$\mathrm{eq}_{1}$}
\left(1-dxu_1\right)f(x) = f(xq) +xu_1 f(xq).
\end{equation}

Then
\begin{equation}
\label{r1}
f(x) = \frac{1+xu_1}{1-dxu_1} f\left(xq\right).
\end{equation}
Iterating~\eqref{r1}, we get
$$f(x)= \prod_{n=0}^{\infty} \frac{1+xu_1q^n}{1-dxu_1q^n} f(0).$$
Thus
$$f(1)=\frac{(-u_1;q)_{\infty}}{(du_1;q)_{\infty}}.$$

Now assume that Theorem~\ref{main} is true for some positive integer $r-1$ and show that it is true for $r$ too.
Let $f$ such that $f(0)=1$ satisfying~\eqref{qdiff}.
Let
$$F(x):= f(x) \prod_{n=0}^{\infty} \frac{1-dxu_rq^n}{1-xq^n}.$$
By Lemma~\ref{lemmaF}, $F(0)=1$ and $F$ satisfies~\eqref{qdiffF}.
Now let $$F(x) =: \sum_{n=0}^{\infty} A_n x^n.$$
Then by Lemma~\ref{lemmaA} $A_0=1$ and $(A_n)_{n \in \N}$ satisfies~\eqref{recA}.
But by Lemma~\ref{equalAA'}, $(A_n)_{n \in \N}$ also satisfies~\eqref{recA'}.
Now let  $$ A_n =: a_n \prod_{k=0}^{n-1} \left( 1 +u_rq^{k} \right).$$
By Lemma~\ref{lemmaA'}, $a_0 =1$ and $(a_n)_{n \in \N}$ satisfies~\eqref{reca}.
Let $$G(x) := \sum_{n=0}^{\infty} a_n x^n.$$
By Lemma~\ref{lemmaa}, $G(0)=1$ and $G$ satisfies~\eqref{qdiffG}.
Finally let $$g(x):= G(x) \prod_{n=0}^{\infty} \left(1-xq^{n}\right).$$
By Lemma~\ref{lemmaG}, $g(0)=1$ and $g$ satisfies $(\mathrm{eq}_{r-1})$.
By the induction hypothesis, we have
\begin{equation}
\label{g1}
g(1)= \prod_{k=1}^{r-1} \frac{(-u_k;q)_{\infty}}{(du_k;q)_{\infty}}.
\end{equation}
By Appell's comparison theorem~\cite{Appell},
\begin{align*}
 \lim_{n \rightarrow \infty} a_n &= \lim_{x \rightarrow 1^-} (1-x) \sum_{n=0}^{\infty} a_n x^n
\\&= \lim_{x \rightarrow 1^-} (1-x) G(x) 
\\&= \lim_{x \rightarrow 1^-} (1-x) \frac{g(x)}{\prod_{n=0}^{\infty} \left(1-xq^{n}\right)}
\\&= \frac{g(1)}{\prod_{n=1}^{\infty}\left(1-q^{n}\right)}.
\end{align*}
Thus
$$\lim_{n \rightarrow \infty} A_n = \prod_{k=0}^{\infty} \left( 1 +u_rq^{k} \right) \frac{g(1)}{\prod_{n=1}^{\infty}\left(1-q^{nN}\right)}.$$
Therefore, by Appell's lemma again,
\begin{equation}
\label{limitF}
\begin{aligned}
\lim_{x \rightarrow 1^-} (1-x) F(x) &= \lim_{n \rightarrow \infty} A_n
\\&= \prod_{k=0}^{\infty} \left( 1 +u_rq^{k} \right) \frac{g(1)}{\prod_{n=1}^{\infty}\left(1-q^{n}\right)}.
\end{aligned}
\end{equation}
Finally,
\begin{align*}
f(1) &= \lim_{x \rightarrow 1^-} f(x)
\\&= \lim_{x \rightarrow 1^-} \prod_{n=0}^{\infty} \frac{1-xq^{n}}{1-dxu_rq^{n}} F(x)
\\&= \frac{\prod_{n=1}^{\infty} \left(1-q^{n}\right)}{ \prod_{n=0}^{\infty} \left(1-du_rq^{n}\right)} \prod_{k=0}^{\infty} \left( 1 +u_rq^{k} \right) \frac{g(1)}{\prod_{n=1}^{\infty}\left(1-q^{n}\right)} \ \text{by~\eqref{limitF}}
\\&= \frac{\left(-u_r;q\right)_{\infty}}{\left(du_r;q\right)_{\infty}} g(1).
\end{align*}
Then by~\eqref{g1},
$$f(1)= \prod_{k=1}^r \frac{(-u_k;q)_{\infty}}{(du_k;q)_{\infty}}.$$
This completes the proof.
\end{proof}

Now Theorem~\ref{refinement} is a simple corollary of Theorem~\ref{main}.

\begin{proof}[Proof of Theorem~\ref{refinement}]
By Lemma~\ref{conj}, $f_{1_{u_1}}$ satisfies~\eqref{qdiff}. Therefore
$$ f_{0_{u_1}}(1) = \prod_{k=1}^r \frac{(-u_k;q)_{\infty}}{(du_k;q)_{\infty}}.$$
This infinite product is the generating function for the overpartitions counted by $D(\ell_1, \dots, \ell_r;k,n)$, and $f_{0_{u_1}}(1)$ the generating function for overpartitions counted by  $E(\ell_1, \dots, \ell_r;k,n)$, thus
$$D(\ell_1, \dots, \ell_r;k,n) = E(\ell_1, \dots, \ell_r;k,n).$$
\end{proof}

\section*{Acknowledgements}
The author thanks Jeremy Lovejoy for his comments on an earlier version of this paper.

\bibliographystyle{alpha}     

\bibliography{references}   

\end{document}